\documentclass{amsart}
\usepackage{amsmath,amssymb,amsfonts,amsthm,mathrsfs,color,graphics,amsxtra,amscd}
\usepackage{hyperref}
\usepackage{mathtools}
\usepackage[all, cmtip]{xy}
 \usepackage[dvips]{graphicx}
\usepackage[a4paper]{geometry}
\usepackage{textcomp}
\usepackage{enumerate}

\newtheorem{thm}{Theorem}[section]
\newtheorem{prop}[thm]{Proposition}
\newtheorem{cor}[thm]{Corollary}
\newtheorem{lem}[thm]{Lemma}

\theoremstyle{definition}
\newtheorem{definition}[thm]{Definition}

\newtheorem{rem}[thm]{Remark}
\newtheorem{art}[thm]{}

\numberwithin{paragraph}{section}
\numberwithin{equation}{thm}

\def\P{{\mathbb P}}
\def\N{{\mathbb N}}
\def\Z{{\mathbb Z}}
\def\Q{{\mathbb Q}}
\def\R{{\mathbb R}}
\def\C{{\mathbb C}}

\def\KO{{\mathcal O}}

\def\KL{{\mathscr L}}
\def\KX{{\mathscr X}}
\def\KY{{\mathscr Y}}

\newcommand{\metr}{{\|\hspace{1ex}\|}}

\def\an{{\rm an}}

\def\div{{\rm div}}

\newcommand{\Xan}{{X^{\rm an}}}
\newcommand{\Yan}{{Y^{\rm an}}}

\newcommand{\Acal}{{\mathscr A}}

\newcommand{\Ecal}{{\mathscr E}}
\newcommand{\Fcal}{{\mathscr F}}

\newcommand{\Hcal}{{\mathscr H}}

\newcommand{\Lcal}{{\mathscr L}}
\newcommand{\Mcal}{{\mathscr M}}

\newcommand{\Ocal}{{\mathcal O}}

\newcommand{\Xcal}{{\mathscr X}}

\newcommand{\Ycal}{{\mathscr Y}}

\newcommand{\Zcal}{{\mathscr Z}}

\newcommand{\Div}{{\rm div}}
\newcommand{\cyc}{{\rm cyc}}
\newcommand{\Pic}{{\rm Pic}}

\newcommand{\Spec}{{\rm Spec}}
\newcommand{\Spf}{{\rm Spf}}

\newcommand{\ve}{{\varepsilon}}
\newcommand{\id}{{\rm id}}

\newcommand{\supp}{{\rm supp}}

\newcommand{\kcirc}{{ K^\circ}}
\newcommand{\ktilde}{{ \tilde{K}}}

\newcommand{\Lan}{{L^{\rm an}}}

\newcommand{\fX}{{\mathfrak{X}}}
\newcommand{\fY}{{\mathfrak{Y}}}
\newcommand{\fU}{{\mathfrak{U}}}
\newcommand{\fL}{{\mathfrak{L}}}
\newcommand{\fM}{{\mathfrak{M}}}

\newcommand{\fV}{{\mathfrak{V}}}
\newcommand{\fW}{{\mathfrak{W}}}

\def\Xcal{{\mathscr X}}

\newcommand{\Ko}{{K^\circ}}

\newcommand{\Pichat}{\widehat{\rm Pic}}

\begin{document}
\title{On Zhang's semipositive metrics}
\author{Walter Gubler and Florent Martin}

\begin{abstract}
Zhang introduced semipositive metrics on a line bundle of a proper variety. 
In this paper, we generalize such metrics for a line bundle $L$ of a  paracompact strictly $K$-analytic space $X$ over any non-archimedean field $K$. 
We prove various properties in this setting such as density of piecewise $\Q$-linear metrics in the space of continuous metrics on $L$. 
If $X$ is proper scheme, then we show that algebraic, formal and piecewise linear metrics are the same. 
Our main result is that on a proper scheme $X$ {over an arbitrary non-archimedean field $K$}, the set of semipositive model metrics is 
closed with respect to pointwise convergence generalizing a result from Boucksom, Favre and Jonsson 
{where $K$ was assumed to be discretely valued with residue characteristic  $0$.}

\bigskip

\noindent
MSC: Primary 14G40; Secondary  14G22
\end{abstract}

\maketitle

\tableofcontents

\section{Introduction}
\label{Introduction}

An arithmetic intersection theory on arithmetic surfaces was introduced by Arakelov and used by Faltings to prove the Mordell conjecture. In higher dimensions, the theory was developed by Gillet and Soul\'e which proved to be a very useful tool in diophantine geometry. To produce arithmetic intersection numbers from a given line bundle $L$ on a proper variety $X$ over a number field $K$, one has to endow the complexification of $L$ with a smooth hermitian metric and one has to choose an  $\Ocal_K$-model $(\Xcal,\Lcal)$ for $(X,L)$. 

Zhang \cite{zhang95} realized that the contribution of a non-archimedean place $v$ to this arithmetic intersection number is completely determined by a metric on $L(K_v)$ associated with $\Lcal$, where $K_v$ is the completion  of $K$ at $v$. This adelic point of view is very pleasant as it allows to deal with archimedean and non-archimedean places in a similar way. 
Motivated by his studies of the Bogomolov conjecture \cite{zhang93}, Zhang \cite{zhang95} introduced semipositive adelic metrics as a uniform limit of metrics induced by nef models and he showed that every polarized dynamical system has a canonical metric inducing the canonical height of Call and Silverman.  

In \cite{gubler98}, it became clear that Zhang's  metrics can be generalized over any non-archimedean field $K$ working with formal models of the line bundle over the valuation ring. It 
turned out that such metrics are continuous on the Berkovich analytification of the line bundle and so we call them continuous semipositive metrics.  

Chambert-Loir introduced measures $c_1(L,\metr)^{\wedge n}$ on the Berkovich space $\Xan$ for a continuous semipositive metric $\metr$ of a line bundle $L$ over $X$ (\cite{chambert-loir-2006}, \cite{gubler-2007b}). 
These measures are non-archimedean equidistribution measures as in  Yuan's equidistribution theorem \cite{yuan-2008} over number fields (see also \cite{CLT09}).
The analogue over function fields was proven in \cite{faber}, \cite{gubler08} and gave rise to progress for the geometric Bogomolov conjecture \cite{gubler-2007b}, \cite{yamaki13,yamaki16}.

Continuous semipositive metrics played an important role in the study of the  arithmetic geometry of toric varieties 
due to Burgos-Gil, Philippon and  Sombra, see \cite{BurgosPhilipponSombra}, \cite{BPS15}, \cite{BPS16}, \cite{BMPS16} with Moriwaki and \cite{BPRS15} with Rivera-Letelier.
Katz--Rabinoff--Zureick-Brown \cite{KRZB} used semipositive model metrics to give  explicit uniform bounds for the number of rational points in situations suitable to the Chabauty--Coleman method.

For the non-archimedean Monge--Amp\`ere problem, continuous semipositive metrics are of central importance. Uniqueness up to scaling was shown by Yuan and Zhang \cite{yuan-zhang}. 
In case of residue characteristic  $0$, a solution was given by Boucksom, Favre and Jonsson \cite{BFJ1,BFJ2} using an algebraicity condition which was removed in \cite{BGJKM}.

Semipositive model metrics also played  a role in the thesis of Thuillier \cite{thuillier05} on potential theory on curves, in the work of Chambert-Loir and Ducros on forms and currents on Berkovich spaces \cite{chambert-loir-ducros} and in the study of delta-forms in \cite{GK1,GK2}. 

Looking at the above references, one observes that the authors work either under the hypothesis that the valuation is discrete or that $K$ is algebraically closed. The reasoning behind the former is that the valuation ring and hence the models are noetherian. If $K$ is algebraically closed, then the valuation ring is not noetherian (unless the valuation is trivial, but we exclude this case here). Working with formal models using Raynaud's theory, this is not really a problem. The assumption that $K$ is algebraically closed is used to have plenty of formal models which have locally the form $\Spf(\Acal^0)$, where $\Acal^0$ is the subring of power bounded elements in an $K$-affinoid algebra $\Acal$. It has further the advantage that finite base changes are not necessary in the semistable reduction theorem or in de Jong's alteration theorems. 
This division has the annoying consequence that many results obtained under one of these hypotheses cannot be used under the other hypothesis. 
Moreover, there is a growing group of people who would like to use Zhang's metrics over any non-archimedean base field\footnote{{The new paper of Boucksom and Eriksson \cite{BE} is an example for this point of view.}}. 
The goal of this paper is to remedy this situation and to study these metrics in the utmost generality which is available to us. 

From now on, we assume that $K$ is a non-archimedean field which means in this paper that $K$ is a field endowed with a non-trivial non-archimedean complete absolute value. 
We denote the valuation ring by $\kcirc$.

We first restrict our attention to the case of a line bundle $L$ on a proper scheme $X$ over $K$. 
We call a metric $\metr$ on $L^\an$ \emph{algebraic} (resp. \emph{formal})  
if it is induced by a   line bundle $\Lcal$ on a flat proper scheme $\Xcal$  
(resp. a line bundle $\fL$ on an admissible formal scheme $\fX$) over $\kcirc$  with generic fibre $X$  and with $L=\Lcal|_X$. 
We use the notation $\metr = \metr_\Lcal$. 
Such a metric is called \emph{semipositive} if $\Lcal$ (resp. $\fL$) 
restricts to a nef line bundle on the special fibre of $\Xcal$ (resp. $\fX$).
More generally, we call $\metr$ a {\it model metric}  if there is a non-zero $k \in \N$ such that $\metr^{\otimes k}$ is an algebraic metric. 
Then a model metric $\metr$ is called {\it semipositive} if $\metr^{\otimes k}$ is semipositive in the previous sense. 
We say that $\metr$ is a {\it continuous semipositive metric} if it is the uniform limit of a sequence of semipositive model metrics on $\Lan$. 

We note that the above definitions are global definitions. 
It is desirable to have local analytic definitions. 
Let $V$ be a paracompact strictly $K$-analytic space and $L$ a line bundle on $V$.
First, we say that a metric $\metr$ on $L$ is a {\it piecewise linear metric} if there is a ${\rm G}$-covering $(V_i)_{i \in I}$ of $V$ 
(i.e. a covering with respect to the {\rm G}-topology on $V$) and frames $s_i$ of $L$ over $V_i$ with $\|s_i\|\equiv 1$. 
Note that such metrics are already considered in \cite{gubler98}, but they were called formal there which is a bit confusing. 
We say that a metric $\metr$ is \emph{piecewise $\Q$-linear} if there is a ${\rm G}$-covering $(V_i)_{i \in I}$ of $V$ and some integers $(k_i)_{i\in I}$ such that for each $i\in I$, the restriction of $\metr^{\otimes k_i}$ to $V_i$ is a piecewise linear metric on $V_i$. 
We refer to Section \ref{section formal and piecewise} for details and  properties.

Following a suggestion of Tony Yue Yu, we call a piecewise linear metric $\metr$ {\it semipositive} at 
 $x  \in V$
if $x$ has  a strictly $K$-affinoid domain 
$W$ of $V$
as a neighborhood (in the Berkovich topology) such that the restriction of $\metr$ to 
$L|_W$
is a semipositive formal metric. 
This notion was studied in \cite{GK2} for $K$ algebraically closed. 
A {\it semipositive piecewise linear metric} on $L$ is a piecewise linear metric which is semipositive at every 
$x \in V$. Semipositive metrics are studied in Section \ref{section Semipositive metrics}.
We highlight here the following result which is useful in comparing the various definitions  mentioned above.

\begin{thm} \label{comparison}
The following are equivalent for a metric $\metr$ on the line bundle $L^\an$ over a proper scheme $X$:
\begin{itemize}
\item[(a)] $\metr$ is an algebraic metric;
\item[(b)] $\metr$ is a formal metric;
\item[(c)] $\metr$ is a piecewise linear metric.
\end{itemize}
The equivalence remains true if we replace ``metric'' by ``semipositive metric'' in every item. 
\end{thm}

As seen in Remark \ref{algebaic vs formal metric}, the equivalence of (a) and (b) follows from 
\cite[Proposition 8.13]{GK1} (as the argument does not use the assumption that $K$ is algebraically closed). The equivalence of (b) and (c) holds more generally over any paracompact strictly $K$-analytic space as shown in Proposition \ref{formal vs piecewise linear metrics}. This equivalence was known before only in case of a compact reduced space over an algebraically closed field. Neither base change nor the old argument can be used and so we give an entirely new argument here. In the semipositive case, the equivalence of (a) and (b) follows immediately from Proposition \ref{model test}.  Finally, the equivalence of (b) and (c) is shown in Proposition \ref{local vs global semipositivity}. It holds more generally for a {separated} paracompact strictly $K$-analytic space.

We also prove the following result (Theorem \ref{theorem density uniform convergence false}) which generalizes \cite[Theorem 7.12]{gubler98} from the compact to the paracompact case.

\begin{thm}
\label{theorem density}
Let $V$ be a paracompact strictly $K$-analytic space with a line bundle $L$.
If $\| \ \|$ is a continuous metric on $L$, then there is a sequence $(\| \ \|_n)_{n\in \N}$ of piecewise $\Q$-linear metrics on $L$ which converges uniformly to $\| \ \|$.
\end{thm}

{Let us come back to semipositive metrics. 
For this, let us consider $X$ a proper scheme over $K$.}
It is a natural question if the notion of semipositivity is closed in the space of model metrics of a given line bundle $L$ of $X$. 
First, we look at this question for uniform convergence of metrics. 
We consider a model metric $\metr$ on $L^\an$ which is semipositive as a continuous metric, which means by definition that it is uniform limit of semipositive model metrics on $L^\an$. 
Then the closedness problem is equivalent to show that $\metr$ is semipositive as a model metric. 
By passing to a tensor power, we may assume that $\metr = \metr_\Lcal$ for a  line bundle $\Lcal$ on a model $\Xcal$ of $X$. By assumption, $\metr$ is the uniform limit of semipositive model metrics $\metr_n$ on $L^\an$. 
For every $n \in \N$, there is a non-zero $k_n \in \N$ such that $\metr_n^{\otimes k_n}$ is an algebraic metric associated with a nef line bundle $\Lcal_n$ living on a proper flat scheme $\Xcal_n$ over $\kcirc$ with generic fiber $X$. Since the models $\Xcal_n$ might be completely unrelated to $\Xcal$, it is non-obvious to show that $\Lcal$ is nef if all the line bundles $\Lcal_n$ are nef. 

An even more challenging problem is to show that the space of model metrics is closed with respect to pointwise convergence. The solution of this problem is the main result of this paper:

\begin{thm}  \label{pointwise convergence of metrics}
Let $X$ be a proper scheme over $K$ with a line bundle $L$. We assume that the model metric $\metr$ on $L^\an$ is a pointwise limit of semipositive model metrics on $L^\an$. Then $\metr$ is a semipositive model metric.
\end{thm}

If {$K$ is discretely valued of  residue characteristic zero} and if $X$ is a smooth projective variety, then this theorem was proven by Boucksom, Favre and Jonsson \cite[Theorem 5.11]{BFJ1} using multiplier ideals. 
They said in \cite{BFJ1} Remark 5.13 that it would be interesting to have a proof along the lines of Goodman's paper \cite[p.178, Proposition 8]{Goo69}. 
This is what we provide in Theorem \ref{pointwise convergence of metrics} with a proof holding for any  non-archimedean field and hence we obtain as an immediate consequence:

\begin{cor} \label{compatibility of semipositive metrics}
A model metric is semipositive as a model metric if and only if it is semipositive as a continuous metric.
\end{cor}

For arbitrary non-archimedean fields, this result was first proven in \cite[Proposition 8.13]{GK2} using a lifting theorem for closed subvarieties of the special fibre. Amaury Thuillier told us that he found a similar (unpublished) lifting argument to prove Corollary \ref{compatibility of semipositive metrics}. 

Theorem \ref{pointwise convergence of metrics} will follow from Theorem \ref{pointwise convergence vs uniform convergence} which is a slightly more general version about pointwise convergence of $\theta$-plurisubharmonic model functions for a closed $(1,1)$-form $\theta$. 
These notions from \cite{BFJ1} will be introduced in Section \ref{section psh model functions}. 
In Theorem \ref{pointwise convergence of metrics} and in Theorem \ref{pointwise convergence vs uniform convergence}, it is enough to require pointwise convergence over all divisorial points of $\Xan$. Such points will be introduced and studied in Appendix \ref{appendix on divisorial points}.

\subsection{Terminology} 
For sets,  in $A\subset B$  equality is not excluded and $A\setminus B$ denotes the complement of $B$ in $A$. 
$\N$ includes $0$. 
All the rings and algebras are commutative with unity. For a ring $A$, the group of units is denoted by $A^\times$. 
If $V$ is a topological space, for a set $U\subset V$ we denote by $U^\circ$ the topological interior of $U$ in $V$.
A variety over a field $k$ is an irreducible and reduced scheme which is separated and of finite type over $k$. 

For the rest of the paper we fix a {\it non-archimedean field} $K$. 
This means here that the field $K$ is equipped with a non-archimedean absolute value $| \ | : K \to \R_+$  which  is complete and non-trivial. Let $v:=-\log |\phantom{a}|$ be the corresponding valuation. 
We have a valuation ring $K^\circ:=\{x\in K\mid v(x)\geq 0\}$ with maximal ideal $K^{\circ \circ}:=\{x\in K\mid  v(x)>0\}$ and residue field $\tilde{K}:=K^\circ / K^{\circ \circ}$. 
We set $\Gamma \coloneqq v(K^\times)$. 
It is a subgroup of $(\R, +)$ called the \emph{value group} of $K$. 
We denote by $\overline{K}$ an algebraic closure of $K$ 
and we set $\C_K$ for the completion of $\overline{K}$. 
It is a minimal algebraically closed non-archimedean field extension of $K$ \cite[\S  3.4.1]{bgr}.

\subsection{Acknowledgements} 
We thank Vladimir Berkovich, Antoine Ducros and Tony Yue Yu for helpful discussions. 
We thank Ofer Gabber for thoroughly answering a question posed by email and we thank the referees for their helpful comments. 
This work was supported by the collaborative research 
center SFB 1085 funded by the Deutsche Forschungsgemeinschaft.

\section{Formal and piecewise linear  metrics}
\label{section formal and piecewise}
For line bundles on paracompact strictly $K$-analytic spaces, we will introduce the global notion of  formal metrics and the local notion of  piecewise linear metrics. We will collect many properties and we will show that both notions agree. At the end, we will prove a density result for piecewise $\Q$-linear metrics.

\begin{art} \label{algebraic models}
Let $X$ be a proper scheme over $K$. Then an {\it algebraic  $\kcirc$-model} of $X$ is a proper flat scheme $\Xcal$ over $\kcirc$ with a fixed isomorphism from the generic fiber $\Xcal_\eta$ to $X$. Usually, we will identify $\Xcal_\eta$ with $X$ along this fixed isomorphism. 
It follows from Nagata's compactification theorem \cite[Theorem 4.1]{Con07}
that an algebraic $\kcirc$-model of $X$ exists. 
The set of isomorphism classes of algebraic $\kcirc$-models of $X$ is partially ordered by morphisms of $\kcirc$-models of $X$ (where by definition such a map extends the identity on $X$). A diagonal argument shows easily that the set of isomorphism classes is directed with respect to this partial order. 

Let $L$ be a line bundle on $X$. An  {\it algebraic $\kcirc$-model} $(\Xcal, \Lcal)$ of $(X,L)$ consists of  an algebraic $\kcirc$-model $\Xcal$ of $X$ and of a line bundle $\Lcal$ on $\Xcal$ with a fixed isomorphism from $\Lcal|_X$ to $L$ which we use again for identification.

It follows from Vojta's version of Nagata's compactification  theorem \cite[Theorem 5.7]{vojta-2007}  and noetherian approximation that $(X,L)$ has always an algebraic $\kcirc$-model. 
Alternatively, one can use the non-noetherian version of Nagata's compactification theorem \cite[Theorem 4.1]{Con07} to get 
an algebraic $K^\circ$-model $\Xcal$ of $X$, then by \cite[Tag 01PI]{stacks-project} one can extend $L$ to an $\mathcal{O}_{\Xcal}$-module of finite presentation $\mathcal{F}$, 
and finally by \cite[Th\'{e}or\`{e}me 5.2.2]{RG71}, 
replacing $\Xcal$  by a dominating $\kcirc$-model, one can ensure that $\mathcal{F}$ is flat, hence a line bundle on $\Xcal$.
\end{art}

\begin{art} \label{formal models}
Let $V$ be a paracompact  strictly $K$-analytic space. We use here the analytic spaces and the terminology introduced by Berkovich in \cite[Section 1]{berkovich-ihes}. 
Then a {\it formal $\kcirc$-model} is an admissible  formal scheme $\fV$ over $\kcirc$ \cite[\S 7.4]{bosch-lectures-2015} with a fixed isomorphism $\fV_\eta \cong V$ on the generic fiber $\fV_\eta$ which we again use for identification. 
Note that we have a canonical reduction map $\pi:V \to \fV_s$ to the special fiber $\fV_s$ (see \cite[Section 2]{GRW2}). 
We say that a covering $(V_i)_{i\in I}$ of $V$ is of \emph{finite type} if for each  $i\in I$, the intersection $V_i \cap V_j$ is nonempty only for finitely many $j \in I$.

The category of paracompact strictly $K$-analytic spaces is equivalent  to  the category of  quasiseparated rigid analytic varieties over $K$ with a strictly $K$-affinoid ${\rm G}$-covering of finite type (see 
\cite[\S 1.6]{berkovich-ihes}) 
and hence we may apply Raynaud's theorem from \cite[Theorem 8.4.3]{bosch-lectures-2015}. 
In particular, we see that a formal $\kcirc$-model of $V$ exists and that the set of isomorphism classes of formal $\kcirc$-models is again directed. 
Some of the references in the following require that $V$ is compact, because the original formulation of Raynaud's theorem in \cite[Theorem 4.1]{bosch-luetkebohmert-1} used that the underlying rigid space is quasicompact and quasiseparated. 
This will be bypassed by using the more general version in  \cite[Theorem 8.4.3]{bosch-lectures-2015} for paracompact $V$ (remember that paracompact includes Hausdorff).

{We always consider the ${\rm G}$-topology on $V$ induced by the strictly $K$-affinoid domains in $V$ (see \cite[\S 1.6]{berkovich-ihes}). The ${\rm G}$-topology is finer than the Berkovich topology and the strictly $K$-affinoid domains may be seen as the basic open subsets of the ${\rm G}$-topology while they are compact in the Berkovich topology of $V$. There is a canonical structure sheaf $\Ocal_{X_{\rm G}}$ on the ${\rm G}$-topology of $V$ such that for every strictly $K$-affinoid domain $W$ of $V$ the corresponding strictly $K$-affinoid algebra is $\Ocal_{X_{\rm G}}(W)$.
A ${\rm G}$-covering $(V_i)_{i \in I}$ is called of {\it finite type} if for every $i \in I$ there are only finitely many $j \in I$ with $V_i \cap V_j \neq \emptyset$.}

Let $L$ be a line bundle on $V$ which means that $L$ is a locally free sheaf of rank $1$ on the ${\rm G}$-topology. 
A {\it formal $\kcirc$-model} $(\fV,\fL)$ of $(V,L)$ consists of a formal $\kcirc$-model $\fV$ of $V$ and a line bundle $\fL$ on $\fV$ with a fixed isomorphism from $\fL|_V$ to $L$ which we use for identification. 
By \cite[Proposition 6.2.13]{chambert-loir-ducros}, a formal $\kcirc$-model of $(V,L)$ always exists. 
\end{art}

\begin{rem} \label{analytification}
If $X$ is a proper scheme over $K$ with a line bundle $L$, then we denote the analytifications by $\Xan$ and $\Lan$ (in the category of Berkovich spaces). By formal completion,  every algebraic $\kcirc$-model $(\Xcal,\Lcal)$ of $(X,L)$ induces a formal $\kcirc$-model $(\hat{\Xcal},\hat{\Lcal})$ of $(\Xan,\Lan)$. 
Note that the special fiber $\Xcal_s$ of $\Xcal$ is canonically isomorphic to the special fiber of the formal completion $\hat{\Xcal}$ and hence the above yields a reduction map $\pi:\Xan \to \Xcal_s$. 
\end{rem}

\begin{lem}
\label{lemma formal vs algebraic}
Let $X$ be a proper scheme over $K$ and let $\fX$ be a formal $K^\circ$-model of $X^{\an}$.
Then there exists  an algebraic $K^\circ$-model $\Xcal$ of $X$ such that $\hat{\Xcal}$ dominates $\fX$.
\end{lem}

\begin{proof}
Let us fix  an algebraic $K^\circ$-model $\Xcal_0$ of $X$.
As recalled in \ref{formal models}, the set of 
isomorphism classes of formal $K^\circ$-models of $X^{\an}$ is a directed set, hence there exists a formal $\kcirc$-model 
$\fV$ which dominates both $\widehat{\Xcal}_0$ and $\fX$. 
Replacing $\fV$ by a larger formal $K^\circ$-model of $\Xan$, it follows from \cite[Lemma 8.4.4 (d)]{bosch-lectures-2015} that the canonical map $\varphi:\fV \to \widehat{\Xcal_0}$ may be assumed to be an admissible formal blowing up in an open coherent ideal $\mathfrak{b}$ of $\widehat{\Xcal_0}$.
Using the formal GAGA-principle  proved by Fujiwara--Kato \cite[Theorem I.10.1.2]{fujiwara-kato-1}, $\mathfrak b$ 
is actually the formal completion of a coherent vertical ideal $\mathfrak{a}$ on $\Xcal_0$.
Hence if $\Xcal$ is the vertical blowing up of $\Xcal_0$ in the ideal $\mathfrak{a}$, by 
\cite[Proposition 8.2.6]{bosch-lectures-2015}  we have $\fV \cong \hat{\Xcal}$ which dominates $\fX$.
\end{proof}

\begin{definition} \label{formal metric}
Let $(\fV, \fL)$ be a formal $\kcirc$-model of $(V,L)$ as in \ref{formal models}. Then we get an associated {\it formal metric} $\metr_\fL$ on $L$ uniquely determined by requiring $\|s\|_\fL=1$ on the generic fibre $W$ of any frame $s$ of $\fL$ over any formal open subset $\fW$ of $\fV$. This is well-defined because a change of frame involves an invertible function $f$ on $\fW$ and we have $|f|=1$ on $W$. 
\end{definition}

\begin{rem} \label{algebaic vs formal metric}
If $(\Xcal,\Lcal)$ is an algebraic $\kcirc$-model of $(X,L)$ as in  \ref{algebraic models}, 
then we get an associated {\it algebraic metric} $\metr_\Lcal$ on $\Lan$ by using the above construction 
for the formal $\kcirc$-model $(\hat{\Xcal},\hat{\Lcal})$ of $(\Xan,\Lan)$ from Remark \ref{analytification}. 
By construction, every algebraic metric is a formal metric. 
The converse is also true as  shown in \cite[Proposition 8.13]{GK1} (as the argument does not use the assumption that $K$ is algebraically closed).
\end{rem}

We have the following extension result from \cite[Proposition 5.11]{GK2}

\begin{prop} 
\label{extension of formal metrics}
Let $L$ be line bundle on a paracompact strictly $K$-analytic space $V$  and let $W$ be a compact strictly $K$-analytic domain of $V$. 
Then every formal metric  on the restriction of $L$ to $W$ extends to a formal metric on $L$.  
\end{prop}

\begin{proof} 
Since this is stated here under more general assumptions than in \cite[Proposition 5.11]{GK2}, we sketch the argument. 
Let $(\fW,\fL)$ be the formal $\kcirc$-model for the given formal metric on $L|_W$. We may assume that $\fW$ is a formal open subset of a formal $\kcirc$-model $\fV$ of $V$ \cite[Lemma 8.4.5]{bosch-lectures-2015}. By the argument in \cite[Lemma 5.7]{bosch-luetkebohmert-1}, there is a coherent $\Ocal_\fV$-module $\Fcal$ on $\fV$ which extends $\fL$. This works even for paracompact $V$  as noted in the proof of \cite[Proposition 6.2.13]{chambert-loir-ducros} 
 and the argument there (or in the proof of \cite[Lemma 7.6]{gubler98}) shows that after replacing $\fV$ by a suitable admissible blowing-up, we may assume that $\Fcal$ is a line bundle. Then the associated  formal metric satisfies the claim.
\end{proof}

\begin{definition} \label{piecewise linear metrics}
Let $V$ be a paracompact strictly $K$-analytic space with a line bundle $L$. 
A metric $\metr$ on $L$ is called {\it piecewise linear} if there is a ${\rm G}$-covering $(V_i)_{i \in I}$ and frames $s_i$ of $L$ over $V_i$ for every $i \in I$ such that $\|s_i\|=1$ on $V_i$.
A function $\varphi \colon V\to \R$ is called a \emph{piecewise linear function} if it induces a piecewise linear metric on the trivial line bundle $\mathcal{O}_V$.
Note that these are ${\rm G}$-local  definitions (see \cite[Proposition 5.10]{GK2} for the argument). 
In particular, these properties are local with respect to the Berkovich topology.
\end{definition}

\begin{lem}
\label{lemma paracompact covering}
Any given $\rm G$-covering of a connected paracompact strictly $K$-analytic space $V$ can be refined to an at most countable $\rm G$-covering $(W_i)_{i \in I}$ of finite type {(see \ref{formal models})} made by strictly $K$-affinoid domains $W_i$. Moreover,  there is always a second  $\rm G$-covering $(U_i)_{i \in I}$ of finite type made by compact strictly $K$-analytic domains $U_i$ such that every $W_i$ is contained in the interior $U_i^\circ$ of $U_i$ with respect to the Berkovich topology.
\end{lem}

\begin{proof}
By \cite[chap. 1, \S 9, Th\'{e}or\`{e}me 5]{BouTop}, $V$ is countable at infinity.
For the proof of the lemma, we assume that $V$ is not compact (the compact case is similar and even easier).
Since compact strictly $K$-analytic domains of $V$ form a basis of neighborhoods of $V$, 
we deduce that there is a sequence $(T_j)_{j\in \N}$  of compact strictly $K$-analytic domains $T_j$ of $V$ 
such that $T_j\subset T_{j+1}^\circ$ for all $j\in \N$ and $V= \cup_{j\in \N} T_j$.

For each $j\in \N$, 
$T_{j+3}^\circ \setminus T_j$ is an open neighborhood of the compact set $T_{j+2} \setminus T_{j+1}^\circ$.
Since compact strictly $K$-analytic domains of $T_{j+3} \setminus T_{j}$ contain a basis of neighborhoods of $T_{j+3} \setminus T_{j}$, 
for each $j\in \N$, we can find $m_j \in \N$ and finitely many 
compact strictly $K$-analytic domains $(T_{j k})_{k=0,\ldots,m_j}$ contained in $T_{j+3}^\circ \setminus T_j$ and covering $ T_{j+2} \setminus T_{j+1}^\circ$.
Since $V = T_1 \cup \bigcup_{j\in \N} (T_{j+2}  \setminus T_{j+1}^\circ)$, we deduce that the covering 
$(V_h)_{h\in H}$ defined after re-indexing the covering 
$\{T_1\}  \cup  \{  T_{j  k} \}_{j \in \N, \ k =0, \ldots , m_j}$ is a countable {\rm G}-covering of finite type 
(to see that $(V_h)_{h\in H}$ is a {\rm G}-covering, one can use \cite[Lemma 1.6.2 (ii)]{berkovich-ihes}).
Since any compact strictly $K$-analytic domain is a finite union of strictly $K$-affinoid domains, we can easily assume that every $V_h$ is actually a strictly $K$-affinoid domain.

Let us first construct the covering $(W_i)_{i \in I}$ refining a given {\rm G}-covering $(Z_l)_{l\in L}$ of $V$. We can replace the latter by a finer {\rm G}-covering and hence we may assume that every $Z_l$ is a  strictly $K$-affinoid domain.
For each index $h\in H$, there are finitely many  
$l_{h1}, \ldots, l_{h q_h} \in L$
 such that the family 
 $\{Z_{l_{h q}} \cap V_h\}_{q=1, \ldots, q_h}$
is a {\rm G}-covering of  {$V_h$}.
{Hence the countable family 
$(Z_{l_{h q}} \cap V_h)_{h\in H, q=1,\ldots ,q_h}$
is a {\rm G}-covering of $V$ of finite type refining $(Z_l)_{l \in L}$. 
By assumption, the underlying topological space of $V$ is Hausdorff and hence it follows from \cite[Theorem 1.6.1]{berkovich-ihes} that the intersection of two strictly $K$-affinoid domains is a finite union of strictly $K$-affinoid domains. We conclude that every $Z_{l_{h q}} \cap V_h$ is a finite union of strictly $K$-affinoid domains. Using them all, we get a {\rm G}-covering $(W_i)_{i \in I}$ of $V$ of  finite type by strictly $K$-affinoid domains $W_i$ refining $(Z_l)_{l \in L}$.}

Finally, for any  {\rm G}-covering $(W_i)_{i \in I}$ of finite type by strictly $K$-affinoid domains $W_i$, we construct a {\rm G}-covering $(U_i)_{i \in I}$ with the required properties.
For $i \in I$ and using the above notations, let $j \in \N$ be the largest number such that $W_i \cap T_j = \emptyset$. If $T_0 \cap W_i$ is non-empty, then we set $j:=-1$ and $T_j:=\emptyset$. 
Since the strictly $K$-analytic domains form a basis of neighborhoods in $V$, there is for every 
$x \in W_i$ a strictly $K$-analytic domain $U_x$ of $V$ such that $U_x$ is a neighborhood of $x$ contained in the complement of $T_j$. Since $W_i$ is compact, it is covered by finitely many $U_x^\circ$. Let $U_i$ be the union of these finitely many $U_x$. Then $U_i$ is a compact strictly $K$-analytic domain of $V$ contained in the complement of $T_j$ and with $W_i \subset U_i^\circ$. 
Since $(W_i)_{i \in I}$ is a {\rm G}-covering refining $(U_i)_{i \in I}$, the latter is also a {\rm G}-covering of $V$.

It remains to prove that $(U_i)_{i \in I}$ is a covering of finite type. We pick $k \in I$. Since $U_k$ is compact, there is a $j \in J$ such that $U_k \subset T_j^\circ$. Since $(W_i)_{i \in I}$ is a {\rm G}-covering, \cite[Lemma 1.6.2 (ii)]{berkovich-ihes} again shows that the compact set $T_j$  is covered by finitely many of the $W_i$. Since $(W_i)_{i \in I}$ is a covering of finite type, we conclude that $T_j$ is intersected by at most finitely many $W_i$. Let us choose any $i \in I$ with $W_i \cap T_j = \emptyset$. By construction of $U_i$, we have $ U_i \cap T_j = \emptyset$ and hence $U_i$ is disjoint from $U_k \subset T_j$. We conclude that the covering $(U_i)_{i \in I}$ is of finite type.
\end{proof}

\begin{prop} \label{formal vs piecewise linear metrics}
Let $\metr$ be a metric on a  line bundle $L$ on a  paracompact strictly $K$-analytic space $V$. Then $\metr$ is formal if  and only if it is piecewise linear.
\end{prop}

\begin{proof} 
{
	Clearly, every formal metric is piecewise linear. Let us prove the converse. For a piecewise linear metric $\metr$ on $L$, there is a ${\rm G}$-covering $(V_i)_{i \in I}$ of $V$ with frames $s_i$ of $L|_{V_i}$ such that $\|s_i\|=1$ on $V_i$. By Lemma \ref{lemma paracompact covering}, we may assume that the ${\rm G}$-covering is of finite type and that every $V_i$ is a strictly $K$-affinoid domain. By Raynaud's theorem \cite[Theorem 8.4.3]{bosch-lectures-2015} and using \cite[Lemma 8.4.5]{bosch-lectures-2015}, there is a formal $K^\circ$-model $\fV$ and a  covering $(\fV_i)_{i \in I}$ of $\fV$ of finite type by quasi-compact formal open subschemes $\fV_i$ with generic fiber $V_i$. Note that any formal $K^\circ$-model is quasi-separated \cite[bottom of p.~204]{bosch-lectures-2015}. 
For every $i,j \in I$, we conclude that the formal open subscheme $\fV_i \cap \fV_j$ is a finite union of formal affine open subschemes $\fV_{ijk}=\Spf(A_{ijk})$. 
For $f_{ij} \coloneqq s_i/s_j \in \Ocal(V_i \cap V_j)^\times$, the identity $\|s_i\| \equiv \|s_j\|$ on $V_i \cap V_j$ yields that   $|f_{ij}| \equiv 1$ on $V_i \cap V_j$. Then \cite[Lemma 8.4.6]{bosch-lectures-2015} shows that $\fV_{ijk}'=\Spf(A_{ijk}[f_{ijk},f_{jik}])$ is an admissible formal scheme and that the canonical morphism  $\fV_{ijk}' \to \fV_{ijk}$ is an admissible formal blowing up.  
By construction, we have $f_{ij} \in \Ocal(\fV_{ijk}')^\times$. 
}

{
We apply now \cite[Proposition 8.2.14]{bosch-lectures-2015} to the covering $\{\fV_{ijk}\}$ of $\fV$ of finite type. This gives the existence of an admissible formal blowing up $\varphi:\fV' \to \fV$ which factorizes through $\fV_{ijk}' \to \fV_{ijk}$ for every $ijk$. We note that $(\varphi^{-1}(\fV_i))_{i \in I}$ is a formal open covering of $\fV'$ of finite type and that
$$\varphi^{-1}(\fV_i) \cap \varphi^{-1}(\fV_j) 
=
\varphi^{-1}(\fV_i\cap \fV_j )
=
\bigcup_k \varphi^{-1}(\fV_{ijk}).$$
Since $\varphi^{-1}(\fV_{ijk})$ is the preimage of $\fV_{ijk}'$ with respect to $\fV_{ijk}' \to \fV_{ijk}$, the above factorization yields $f_{ij} \circ \varphi \in \Ocal((\varphi^{-1}(\fV_{ijk}))^\times$ for every $ijk$ and hence $f_{ij} \circ \varphi \in \Ocal(\varphi^{-1}(\fV_i) \cap \varphi^{-1}(\fV_j))^\times$. Let $\Lcal$ be the model of $L$ on $\fV'$ given by the transition functions $f_{ij} \circ \varphi$ with respect to the covering $(\varphi^{-1}(\fV_i))_{i \in I}$. Then the construction shows that $\metr = \metr_\Lcal$.}
\end{proof}

\begin{definition} \label{Q-formal} 
Let $V$ be a paracompact strictly $K$-analytic space with a line bundle $L$.
A metric $\metr$ on  $L$  is called {\it piecewise $\Q$-linear} if for every $x\in V$ there exists  an open neighborhood $W$ of $x$ and a non-zero $n \in \N$ such that $\metr^{\otimes n}_{|W}$ is a  piecewise linear metric on $L^{\otimes n}_{|W}$. 
A function $\varphi \colon V\to \R$ is called a \emph{piecewise $\Q$-linear function} if it induces a piecewise $\Q$-linear metric on the trivial line bundle $\mathcal{O}_V$.
\end{definition}

\begin{prop} \label{properties for piecewise linear}
Let $V$ be a paracompact strictly $K$-analytic space with a line bundle $L$. Then the following properties hold:
\begin{itemize}
\item[(a)] A piecewise $\Q$-linear metric on $L$ is continuous.
\item[(b)] The isometry classes of piecewise linear (resp.\ piecewise $\Q$-linear) metrics on line bundles of $V$ form an abelian group with respect to $\otimes$.
\item[(c)]  The pull-back $f^*\metr$ of a piecewise linear (resp.\ piecewise $\Q$-linear) metric $\metr$ on $L$ with respect to a morphism $f:W \to V$ of paracompact strictly $K$-analytic spaces is a piecewise linear (resp.\ piecewise $\Q$-linear) metric on $f^*L$.
\item[(d)] The minimum and the maximum of two piecewise linear (resp.\ piecewise $\Q$-linear) metrics on $L$ are again 
 piecewise linear (resp.\ piecewise $\Q$-linear) metrics on $L$.
\end{itemize}
\end{prop}

\begin{proof} 
These properties are proved in \cite[Section 7]{gubler98} under the assumption that $K$ is algebraically closed and $V$ is compact. The assumption $K$ algebraically closed was not used in the arguments. Since (a)--(d) are local statements, we can deduce them from the corresponding statements in {\it loc.~ cit.} 
\end{proof}

Let $V$ be a paracompact strictly $K$-analytic space. 
Recall that for $U \subset V$, we denote the topological interior of $U$ in $V$ by $U^\circ$.

\begin{lem}
\label{lemma extension}
Let $W \subset U \subset V$ where $W,U$ are compact strictly $K$-analytic domains of $V$ with $W \subset U^\circ$.
Let $f \colon W \to \R$ be a piecewise linear function.
Then $f$ extends to a piecewise linear function $\varphi \colon V \to \R$ such that $\supp(\varphi) \subset U$.
\end{lem}

\begin{proof}
By compactness of $U \setminus U^\circ$, there exists a compact strictly $K$-analytic domain $Z \subset V$ such that $Z$ is a neighborhood of $U\setminus U^\circ$ and 
$W\cap Z=\emptyset$.
Hence $W \coprod Z$ is a compact strictly $K$-analytic domain of $V$ and we consider the piecewise linear function on $W \coprod Z$ defined by $f$ on $W$ and by $0$ on $Z$.
Then we apply Proposition \ref{extension of formal metrics} to $L= \mathcal{O}_V$, in which case formal metrics correspond to piecewise linear functions (see Proposition \ref{formal vs piecewise linear metrics}). 
We deduce that there exists a piecewise linear function $ g \colon V \to \R$ which agrees with $f$ on $W$ and which agrees with $0$ on $Z$.
But since $Z$ is a neighborhood of $ U\setminus U^\circ$, we deduce that the function $\varphi \colon V \to \R$ defined by 
\[ \varphi (x) = 
\begin{cases} 
g(x) & {\rm if} \ x\in U \\
0 & {\rm if } \ x\notin U
\end{cases}
\]
is still piecewise linear as this is a local property.
Since $\varphi$ extends $f$ and $\supp(\varphi) \subset U$, we get the claim.
\end{proof}

\begin{lem}
\label{lemma extension 2}
Let $V$ be a paracompact strictly $K$-analytic space.
Let $W \subset V$  be a compact strictly $K$-analytic domain of $V$ and let $f \colon W \to \R$ be a continuous function with $f\geq 0$.
Then for any $\varepsilon >0$ there exists a piecewise $\Q$-linear function $\varphi$ on $V$ such that $\varphi \geq 0$ and 
for all $x\in W$ we have  $f(x) -\varepsilon \leq \varphi(x) \leq f(x)$.
\end{lem}

\begin{proof}
Since piecewise $\Q$-linear functions are dense in the compact case \cite[Theorem 7.12]{gubler98}, there exists a piecewise $\Q$-linear function $g \colon W \to \R$ such that 
$f- \varepsilon \leq g \leq f$ on $W$.
Since $W$ is compact, there is a non-zero $k \in \N$ such that $kg$ is piecewise linear. 
By Proposition \ref{extension of formal metrics} and Proposition \ref{formal vs piecewise linear metrics} applied to the formal metric on $\Ocal_V$ associated with $kg$, there exists a piecewise $\Q$-linear function $\psi \colon V \to \R$ which extends $g$.
We then set $\varphi \coloneqq \max(\psi, 0)$.
By Proposition \ref{properties for piecewise linear} (d), $\varphi$ is piecewise $\Q$-linear.
By definition, we have $\varphi \geq 0$. 
We have $\psi \leq f$ on $W$ and $f$ is non-negative,  hence 
 we have $\varphi \leq f$ on $W$.
Finally, since $f-\varepsilon \leq \psi$ on $W$ we also have that 
$f-\varepsilon \leq \max(\psi,0) = \varphi$ on $W$.
\end{proof}

\begin{prop}
\label{proposition density uniform convergence}
Let $V$ be a paracompact strictly $K$-analytic space.
Let $f \colon V \to \R$ be a continuous function on $V$.
Then $f$ can be uniformly approximated by piecewise $\Q$-linear functions.
In other words, for every $\varepsilon >0$ there exists a piecewise $\Q$-linear function $\varphi \colon V\to \R$ such that 
$\sup_{x\in V} |f(x)-\varphi(x) | \leq \varepsilon$.
\end{prop}

\begin{proof}
We will use that the result holds when $V$ is compact \cite[Theorem 7.12]{gubler98}. 
Note that  in \cite[\S 7]{gubler98},  $K$ was assumed to be algebraically closed, but the
argument for \cite[Theorem 7.12]{gubler98} does not use this assumption and so we can use the result over any non-archimedean field.
Let $f_+ \coloneqq \max(f,0)$ and $f_-\coloneqq \max(-f,0)$ so that $f=f_+ - f_-$.
Hence replacing $f$ by $f_+$ or $f_-$ we can assume that $f\geq 0$.

We can work separately on the connected components of $V$, hence we may assume that $V$ is connected.
By Lemma \ref{lemma paracompact covering}, we can find $(W_i)_{i\in I}$ and $(U_i)_{i\in I}$ two $G$-coverings of $V$ of finite type  with a finite or countable $I$ and 
$W_i \subset U_i^\circ$ for all $i\in I$.
In the following, we assume $I=\N \setminus \{0\}$. The finite case is similar and easier.
Let us now fix $\varepsilon >0$ and let us construct a family of piecewise $\Q$-linear functions $(\varphi_i)_{i\in I}$ 
with $\varphi_i : V \to \R$ such that 
\begin{enumerate}[(i)]
\item for all $i\in I$, $\supp(\varphi_i) \subset U_i$ and $\varphi_i \geq 0$.
\item for all $n\in I$ we have 
$f \geq \sum_{i=1}^n  \varphi_i  \geq f-\varepsilon$ on $\cup_{i=1}^n W_i$.
\item  $f \geq \sum_{i=1}^n \varphi_i $ on $V$.
\end{enumerate}
Observe that this will conclude the proof of the proposition since then $\varphi \coloneqq \sum_{i\in I} \varphi_i$ is a well defined piecewise $\Q$-linear function such that 
$|f - \varphi| \leq \varepsilon$.
The rest of the proof is dedicated to construct inductively a family $(\varphi_i)_{i\in I}$ satisfying the conditions (i), (ii) and (iii).

Let us consider $n\geq 1$ and let us assume that we are given piecewise $\Q$-linear functions $\varphi_1,\ldots,\varphi_n$ satisfying the above conditions.
We will now construct a piecewise $\Q$-linear function $\varphi_{n+1}$ such that $\varphi_1,\ldots, \varphi_{n+1}$ satisfies the conditions (i), (ii) and (iii).

By the density result in the compact case \cite[Theorem 7.12]{gubler98}, 
we know that there exists a piecewise $\Q$-linear function $g \colon W_{n+1} \to \R$ such that 
\begin{equation}
\label{eq density 0}
f-\sum_{i=1}^n \varphi_i - \varepsilon \leq g \leq f-\sum_{i=1}^n \varphi_i \hspace{20pt}  {\rm on } \ W_{n+1}
\end{equation}
Then by Lemma \ref{lemma extension} applied to $g$ and $W_{n+1} \subset U_{n+1} \subset V$, there exists a 
piecewise $\Q$-linear function $\Psi \colon V \to \R$ which extends $g$ and with $\supp(\Psi)\subset U_{n+1}$.
Then \eqref{eq density 0} becomes 
\begin{equation}
\label{eq density 1}
f-\varepsilon \leq \Psi + \sum_{i=1}^n \varphi_i \leq f \hspace{20pt} {\rm on } \ W_{n+1}.
\end{equation}
Then we set
\[ \psi \coloneqq \max(0, \Psi).\]
From this definition, we get that $\supp(\psi) \subset \supp(\Psi) \subset U_{n+1}$.
It is a piecewise $\Q$-linear function by Proposition \ref{properties for piecewise linear} (d) and it satisfies $\psi \geq 0$.
Now, \eqref{eq density 1} combined with the condition (iii) for $n$ yields
\begin{equation}
\label{eq density 1 bis}
\psi + \sum_{i=1}^n \varphi_i \leq f \hspace{20pt} \ {\rm on }\ W_{n+1}.
\end{equation}
Also, since $\Psi \leq \psi$, we deduce from \eqref{eq density 1} that 
\begin{equation}
\label{eq density 1 bisbis}
f-\varepsilon \leq \psi + \sum_{i=1}^n \varphi_i \hspace{20pt} {\rm on} \ W_{n+1}.
\end{equation}
On the other hand, since $\psi \geq 0$, the condition (ii) for $n$ yields 
\begin{equation}
\label{eq density 1 bisbisbis}
f-\varepsilon \leq \psi +  \sum_{i=1}^n \varphi_i \ \hspace{20pt}  \ {\rm on } \ \bigcup_{i=1}^n W_i.
\end{equation}
From 
 \eqref{eq density 1 bisbis} and \eqref{eq density 1 bisbisbis}, we deduce that 
\begin{equation}
\label{eq density 2}
f-\varepsilon \leq \psi +  \sum_{i=1}^n \varphi_i 
\hspace{20pt} \ {\rm on } \ \bigcup_{i=1}^{n+1} W_i.
\end{equation}
Lemma \ref{lemma extension 2} applied to the non negative function $f-\sum_{i=1}^n \varphi_i \colon V \to \R$ and 
to the compact $K$-analytic domain $\cup_{i=1}^{n+1}U_i$ yields a piecewise $\Q$-linear function 
$\chi \colon V \to \R$ such that $\chi \geq 0$ and 
\begin{equation}
\label{eq density 3}
f -  \sum_{i=1}^n \varphi_i  - \varepsilon \leq \chi \leq f -\sum_{i=1}^n \varphi_i \hspace{20pt} \ {\rm on } \ \bigcup_{i=1}^{n+1} U_i.
\end{equation}
We then set 
\[ \varphi_{n+1} \coloneqq \min(\psi, \chi).\]
By Proposition \ref{properties for piecewise linear} (d), $\varphi_{n+1}$ is a piecewise $\Q$-linear function.
Since $\psi \geq 0$ and $\chi \geq 0$ we get that $\varphi_{n+1}\geq 0$ and we also get that for $x\in V$, $\psi(x)= 0 \Rightarrow \varphi_{n+1}(x)=0$. 
This implies that $\supp (\varphi_{n+1}) \subset \supp (\psi) \subset U_{n+1}$.
Hence (i) is satisfied for $\varphi_{n+1}$.

Let us now prove that 
\begin{equation}
\label{eq density 4}
\sum_{i=1}^{n+1} \varphi_i \leq f \ \ {\rm on } \ V.
\end{equation}
Let $x\in V$. 
We first  suppose that $x\in U_{n+1}$.
Then by \eqref{eq density 3}, we have 
$\chi(x) + \sum_{i=1}^n \varphi_i(x)  \leq f(x)$.
By definition of $\varphi_{n+1}$, we have 
$\varphi_{n+1}\leq \chi$ hence 
\[ \varphi_{n+1}(x) + \sum_{i=1}^n \varphi_i(x) \leq \chi(x) + \sum_{i=1}^n \varphi_i(x) \leq f(x).\]
If $x\notin U_{n+1}$, then we have $\psi(x) =0$ since $\supp(\psi) \subset U_{n+1}$, hence $\varphi_{n+1}(x) = 0$.
So by the condition (iii) for $n$, we get
\[ \sum_{i=1}^{n+1} \varphi_i(x) = \sum_{i=1}^n \varphi_i(x) \leq f(x).\]   
This proves \eqref{eq density 4}, whence condition (iii) holds for $n+1$.

Let us finally prove that 
\[
f -\varepsilon \leq \sum_{i=1}^{n+1} \varphi_i \leq f \hspace{30pt} {\rm on} \ \bigcup_{i=1}^{n+1} W_i.
\]
The right inequality has been proven in \eqref{eq density 4} so it only remains to prove the left inequality.
By \eqref{eq density 2}, we have 
\begin{equation}
\label{eq density 5}
f-\varepsilon \leq \psi + \sum_{i=1}^n \varphi_i \hspace{30pt} {\rm on} \ \bigcup_{i=1}^{n+1}W_i
\end{equation}
and by construction (see \eqref{eq density 3} having in mind that $W_i \subset U_i$), we have 
\begin{equation}
\label{eq density 6}
f-\varepsilon \leq \chi + \sum_{i=1}^n \varphi_i \hspace{30pt} {\rm on} \ \bigcup_{i=1}^{n+1} W_i.
\end{equation}
Hence \eqref{eq density 5} and \eqref{eq density 6} yield that 
\[ f-\varepsilon \leq \min(\psi,\chi) +\sum_{i=1}^n \varphi_i  = \sum_{i=1}^{n+1} \varphi_i \hspace{20pt} {\rm on} \ \bigcup_{i=1}^{n+1} W_i\]
which proves condition (ii) for $\varphi_1,\ldots,\varphi_{n+1}$.
By induction, this proves the existence of a family $(\varphi_i)_{i\in I}$ satisfying conditions (i), (ii) and (iii).
\end{proof}

\begin{rem}
\label{remark partition of unity}
The proof of Proposition \ref{proposition density uniform convergence} also gives that if  $\varphi \colon V \to \R$ is a piecewise $\Q$-linear function on a  paracompact strictly $K$-analytic space $V$, 
then there exists a family $(\varphi_i)_{i\in I}$ of piecewise $\Q$-linear functions on $V$ such that the family $\supp(\varphi_i)_{i\in I}$ is a locally finite family of compact sets 
subordinate to any given open covering of $V$
and such that 
$\varphi = \sum_{i\in I} \varphi_i$.
Indeed, in the above proof we may construct the covering $U_i$ finer than the given open covering and then we may use $\varepsilon=0$ 
in the construction due to piecewise $\Q$-linearity.
\end{rem}

\begin{thm}
\label{theorem density uniform convergence false}
Let $V$ be a paracompact strictly $K$-analytic space with a line bundle $L$.
If $\| \ \|$ is a continuous metric on $L$, then there is a sequence $(\| \ \|_n)_{n\in \N}$ of piecewise $\Q$-linear metrics on $L$ which converges uniformly to $\| \ \|$.
\end{thm}
\begin{proof}
We have seen at the end of \ref{formal models}
that $L$ admits a formal metric.
Hence, tensoring by $L^{-1}$, we can assume that $L= \mathcal{O}_V$ 
and we are reduced to prove that for any continuous function $f \colon V \to \R$ there exists a sequence of 
 {piecewise $\Q$-linear} functions $(\varphi_n)_{n\in \N}$ which converges uniformly to $f$
which was done in Proposition \ref{proposition density uniform convergence}.
\end{proof}

The next result deals with base change of piecewise linear metrics. We denote by {$\hat{\otimes}_K F$} the base change functor from the base field $K$ to a non-archimedean extension field $F$ applied to the category of strictly $K$-analytic spaces or to the line bundles on such spaces. The argument for (b)  is due to Yuan (see \cite[Lemma 3.5]{yuan-2008}).

\begin{prop} \label{metric and base extension}
Let $L$ be a line bundle on a paracompact strictly $K$-analytic space $V$ and let $F/K$ be a non-archimedean field extension. 
\begin{itemize}
\item[(a)] The base change of a piecewise linear (resp.\ piecewise $\Q$-linear) metric on $L$  is a piecewise linear (resp.\ piecewise $\Q$-linear)  metric on $L \hat{\otimes}_K F$.
\item[(b)] If $F$ is a subfield of $\C_K$ and if $V$ is compact, then every piecewise linear (resp.\ piecewise $\Q$-linear) metric on  $L \hat{\otimes}_K F$ is the base change of a unique piecewise linear (resp.\ piecewise $\Q$-linear) metric on $L \hat{\otimes}_K {K'}$ for a suitable finite subextension $K'/K$ of $F/K$.
\end{itemize}
\end{prop}

\begin{proof} 
It follows from \cite[Theorem 1.6.1]{berkovich-ihes} that the base change of $V$ to $F$ is a paracompact strictly $F$-analytic space. 
Property (a) is obvious. 

To prove (b), we assume that $\metr$ is a piecewise linear metric on $L \hat{\otimes}_K F$. 
We have seen in \ref{formal models} that $(V,L)$ has a formal $\kcirc$-model $(\fV,\fL)$ and so we may assume that $L=\Ocal_V$ by passing to $\metr/\metr_{\fL \hat{\otimes}_{\kcirc} F^\circ}$.  By Proposition \ref{formal vs piecewise linear metrics}, there is a formal $F^\circ$-model $(\fV'',\fL'')$ of $(V \hat{\otimes}_K F,L \hat{\otimes}_K F)$ such that $\metr = \metr_{\fL''}$. By Raynaud's theorem \cite[Theorem 4.1]{bosch-luetkebohmert-1}, we may assume that there is an admissible formal blowing up $\fV'' \to \fV \hat{\otimes}_{\kcirc} F^\circ$. Note that $L=\Ocal_V$ yields that $\fL''= \Ocal(E)$ for a vertical  Cartier divisor $E$ on $\fV''$. Replacing $\metr$ by a suitable multiple, we may assume that $E$ is an effective Cartier divisor. 

An approximation argument based on the density of the algebraic closure of $K$ in $F$ shows that the coherent ideal of the admissible formal blowing up 
is defined over $(K')^\circ$ for a  finite subextension $K'/K$ of $F/K$. 
We conclude  that  $\fV'' \to \fV \hat{\otimes}_{\kcirc} F^\circ$ is the base change of an admissible formal blowing up $\fV' \to \fV \hat{\otimes}_{\kcirc} (K')^\circ$ for a formal $(K')^\circ$-model $\fV'$ of $V \hat{\otimes}_K K'$.  
We choose a finite  covering $(\fU_i')_{i \in I}$ of $\fV'$ by formal affine open subsets $\fU_i'$ of $\fV'$. 
Then the coherent sheaf of ideals $\Ocal(-E)$ restricted to $\fU_i' \hat{\otimes}_{(K')^\circ} F^\circ$ is generated by finitely many regular functions. A similar approximation argument as above shows that all these generators can be replaced by regular functions on $\fU_i'$ if we replace $K'$ by a larger finite subextension  of $F/K$. We conclude that $\fL''=\Ocal(E)$ is defined on $\fV'$ proving (b). Note that uniqueness is obvious.  
\end{proof}

\section{Semipositive metrics}
\label{section Semipositive metrics}
We will first introduce semipositive formal metrics. We have seen in Proposition \ref{formal vs piecewise linear metrics} that formal metrics are the same as piecewise linear metrics and hence everything applies to piecewise linear metrics as well. 

\begin{art} \label{definition semipositive algebraic metrics}
Let $X$ be a proper scheme over $K$ with a line bundle $L$ over $X$. We call an algebraic $\kcirc$-model $(\Xcal,\Lcal)$ of $(X,L)$  
{\it numerically effective} (briefly {\it nef}) if $\deg_\Lcal(C) \geq 0$ for every closed curve $C$ in $\Xcal$ which is proper 
over $\kcirc$. Of course, properness implies that $C$ is contained in the special fiber $\Xcal_s$. 
An algebraic metric $\metr$ on $L^\an$ is said to be {\it semipositive} if there is  a nef algebraic $\kcirc$-model $(\Xcal,\Lcal)$ of $(X,L)$ such that $\metr = \metr_\Lcal$. 
We say that a line bundle $\Lcal$ on $\Xcal$ is {\it numerically trivial} if $\deg_\Lcal(C) = 0$ for every closed curve $C$ in $\Xcal$ which is proper 
over $\kcirc$. Equivalently, we can require that $\Lcal$ and $\Lcal^{-1}$ are both nef. We say that $\Lcal$ is {\it numerically equivalent} to a line bundle $\Lcal'$ on $\Xcal$ if $\Lcal' \otimes \Lcal^{-1}$ is numerically trivial.
\end{art}

\begin{art} \label{definition semipositive formal metrics}
The above definition is easily generalized to the analytic setting: Let $L$ be a line bundle on a paracompact  strictly $K$-analytic space $V$. 
A formal $\kcirc$-model $(\fV,\fL)$ of $(V,L)$ is called {\it nef} if $\deg_\fL(C) \geq 0$ for any closed curve $C$ in the special fiber $\fV_s$ which is proper over the residue field $\ktilde$. 
A formal metric $\metr$ on $L$ is called {\it semipositive} if there is a nef formal $\kcirc$-model $(\fV,\fL)$ of $(V,L)$ such that $\metr = \metr_\fL$. 
Obviously, the trivial metric on $\Ocal_V$ is a semipositive formal metric.
\end{art}

It will follow from Proposition \ref{model test} below that we may use any model to test semipositivity of the associated metrics. 
Based on this result, it is easy to check that the tensor product of two semipositive formal metrics is again a semipositive formal metric.

\begin{lem}
\label{lemma semipositive base change}
Let $V$ be a paracompact strictly $K$-analytic space, $L$ a line bundle on $V$ and $(\fV,\fL)$ a formal $K^\circ$-model of $(V,L)$.
Let $F$ be a non-archimedean extension of $K$ and 
$(\fV_F,\fL_F)$ the formal $F^\circ$-model of $(V_F,L_F)$ obtained by base change. 
Then $\fL$ is nef if and only if $\fL_F$ is nef. 
\end{lem}

\begin{proof}
We remark that $\fV_s \otimes_{\tilde{K}} \tilde{L}   \cong   (\fV \hat{\otimes}_{\Ko} L^\circ)_s$.
Hence the result follows from the fact that a line bundle  on a proper scheme over $\ktilde$ is nef if and only if its pull-back to $\tilde{F}$ is nef.
This is proven in the projective case in \cite[Remark 1.3.25]{DFEM} and the proper case follows from Chow's lemma and the projection formula.
\end{proof}

\begin{lem}
\label{semipositive reduced}
Let $V$ be a paracompact strictly $K$-analytic space, $L$ a line bundle on $V$ and $(\fV,\fL)$ a formal $K^\circ$-model of $(V,L)$. 
Let $(\fV_{\rm red},\fL_{\rm red})$ be the formal $K^\circ$-model of $(V_{\rm red},L_{\rm red})$ obtained by putting the induced reduced structure.
Then $\fL$ is nef if and only if $\fL_{\rm red}$ is nef. 
\end{lem}

\begin{proof}
Let $\fV_{\rm red}$  be the induced  reduced structure on $\fV$. 
Since $\fV_{\rm red} \to \fV$ is finite (in fact a closed immersion), we deduce that the induced map $(\fV_{\rm red})_ s \to \fV_s$ between the special fibers is finite.  
By the projection formula, we conclude  that $\fL$ is nef if and only if $\fL_{\rm red}$ is nef.
\end{proof}

\begin{prop} 
\label{model test}
Let $(\fV,\fL)$ be a formal $\kcirc$-model of $(V,L)$. 
Then $\metr_\fL$ is a semipositive formal metric if and only if $\fL$ is a nef formal $\kcirc$-model.  
\end{prop}

\begin{proof}
By definition if $\fL$ is nef, then $\| \ \|_\fL$ is semipositive, so we only have to prove the reverse implication. 
Hence we assume that $\metr_\fL$ is a semipositive formal metric and we have to show that $\fL$ is nef.
Using Lemma \ref{lemma semipositive base change}, we can replace $K$ by $\C_K$ and hence we may assume that $K$ is algebraically closed.

By definition of semipositivity, there is a nef formal $\kcirc$-model $\fM$ of $L$ on some formal $\kcirc$-model $\fW$ of $V$ with $\metr_\fL = \metr_\fM$. 
There exists  a model $\fX$ of $V$ which dominates both $\fV$ and $\fW$. 
Let $\pi : \fX \to \fV$ be the induced morphism.
Since the induced morphism on the special fibers $\pi_s \colon \fX_s \to \fV_s$ is proper and surjective, by the projection formula, 
$\fL$ is nef if and only $\pi^* \fL$ is nef. 
Hence replacing $(\fV,\fL)$ by $(\fX, \pi^* \fL)$, we can assume that $\fV$ dominates $\fW$.

Let $\fV_{\rm red}$ be the induced reduced structure on $\fV$. 
Hence $\fV_{\rm red} \to \fV$ is finite.
Locally, $\fV_{ \rm red}$ is given by $\Spf(A)$ for some reduced admissible $\Ko$-algebra $A$.
Let $\Acal \coloneqq A\otimes_{\Ko} K$. 
It is a strictly $K$-affinoid algebra, and by \cite[6.4.3]{bgr} $A' \coloneqq \Acal^\circ$ is an admissible $K^\circ$-algebra as $K$ is algebraically closed, and  by {\cite[Proposition 3.8]{baker-payne-rabinoff}}, $A \to A'$ 
is finite.
By definition of $A'$ we have an isomorphism $A \hat{\otimes}_\kcirc K \cong A' \hat{\otimes}_\kcirc K \cong \Acal$. 
By \cite[Proposition 7.2.6/3]{bgr}, we can glue the morphisms $\Spf(A') \to \Spf(A)$ to get a formal $\kcirc$-model $\fV'$ of $V_{\rm red}$ such that 
$\fV' \to \fV_{\rm red}$ is finite.
In particular, we deduce that the induced morphisms $\fV'_s \to (\fV_{\rm red})_s \to \fV_s$  are proper and surjective, and we conclude from the projection formula that $\fL$ is nef if and only if  its pull-back $\fL'$ to $\fV'$ is nef.

By construction, $\fV'$ is locally of the form $\Spf(\Acal^\circ)$, hence we deduce that $\fV'_s$ is locally given by $\Spec(\tilde{\Acal})$ 
which is reduced.
Now we use the fact that on an admissible formal scheme with reduced special fibre and with $K$ algebraically closed, the 
metric $\metr_{\fL'}$ determines the model $\fL'$ up to isomorphism (see \cite[Proposition 7.5]{gubler98}).
Using that $\metr_{\fM'}=\metr_\fL= \metr_{\fL'}$ for the pull-back $\fM'$ of $\fM$ to $\fV'$,   we deduce that $\fM' \cong \fL'$.
As above, the pull-back $\fM'$ of $\fM$ is nef and hence $\fL'$ is nef.
\end{proof}

\begin{prop} \label{pullback of semipositive formal}

Let $f:V'\to V$ be a morphism of paracompact strictly $K$-analytic spaces and let $\metr$ be a formal metric on a line bundle $L$ of $V$. 
\begin{itemize}
\item[(a)] If $\metr$ is a semipositive formal metric, then $f^*\metr$ is a semipositive formal metric. 
\item[(b)] If $f$ is a surjective proper morphism and if $f^*\metr$ is a semipositive formal metric, then $\metr$ is a semipositive formal metric.
\end{itemize}
\end{prop}

\begin{proof}
There is a formal $K^\circ$-model $(\fV,\fL)$ of $(V,L)$ such that $\metr= \metr_\fL$. 
We use Raynaud's theorem to extend $f$ to a morphism $\varphi:\fV' \to \fV$ of formal $\kcirc$-models. 
Then 
\begin{equation} \label{pullback and formal}
f^*\metr = \metr_{\varphi^*\fL}
\end{equation}
shows that $f^*\metr$ is a formal metric. 
To check semipositivity, Lemma \ref{lemma semipositive base change} shows that $K$ may be assumed to be algebraically closed which will allow us to use the results of \cite{Kle}.
Then the claims follow from \cite[Proposition I.4.1]{Kle} 
applied to $\varphi_s$. 
For (b), we use additionally that $\varphi_s$ is proper by \cite[Corollary 4.4]{Tem} and surjective 
(as $f$ and the reduction map $V \to \fV_s$ are surjective). 

\end{proof}

\begin{lem}
\label{irreducible components}
Let $X$ be a proper scheme over $K$, $L$ a line bundle on $X$ and $(\KX,\KL)$ an algebraic $K^\circ$-model of $(X,L)$ with $\| \ \| \coloneqq \| \ \|_\KL$.
Let $(X_i)_{i\in I}$ be the irreducible components of $X$ equipped with their induced reduced structures.
Then $\| \ \|$ is semipositive if and only if  $\| \ \|_{|X_i}$ is semipositive for all $i \in I$.
\end{lem}
\begin{proof}
For each $i\in I$, let $\KX_i$ be the closed subscheme of $\KX$ defined as the topological closure of $X_i$ in $\KX$ equipped with the induced reduced structure. 
We then get for each $i\in I$  a cartesian diagram 
\[ 
\xymatrix{
X_i \ar[r]  \ar[d] &  \ar[d] X \\
\KX_i \ar[r] & \KX
}
\]
Since the morphism $\coprod_{i\in I} \KX_i \to \KX$ is finite surjective, 
 the projection formula shows that $\KL$ is nef on $\KX$ if and only if $\KL|_{\KX_i}$ is nef on $\KX_i$ for all $i$.
\end{proof}

\begin{art} \label{local definition of semipositive formal metrics}
Following a suggestion of Tony Yue Yu, we can define semipositivity locally on $V$. 
We say that a piecewise linear metric on $L$ is {\it semipositive at $x \in V$} if there is 
a compact strictly $K$-analytic domain $W$ in $V$ which is a neighborhood of $x$ such that the restriction of $\metr$ to $L|_W$ is a semipositive formal metric in the sense of \ref{definition semipositive formal metrics} 
(using the equivalence of Proposition \ref{formal vs piecewise linear metrics}).
We say that $\metr$ is {\it a semipositive piecewise linear metric} if it is semipositive at all $x \in V$. 
We will see  in Proposition \ref{local vs global semipositivity} that this fits with the definition in \ref{definition semipositive formal metrics} 
assuming that $V$ is {separated}. 
\end{art} 

\begin{definition}
\label{semipositive Q-formal}
Let $\metr$ be a  piecewise $\Q$-linear metric on the line bundle $L$ over $V$ and let  $x \in V$.
Then $\metr$ is called {\it semipositive} at $x \in V$ if and only if we may choose a compact strictly $K$-analytic domain  $W$ which is  a neighborhood of $x$ 
and some integer $k\geq 1$ such that $\metr_{|W}^{\otimes k}$ is a semipositive formal metric.
\end{definition}

It follows easily from Proposition \ref{model test} that a piecewise linear  
metric on $L$ is semipositive as a  piecewise linear  
metric if and only if it is semipositive as a piecewise $\Q$-linear  
metric.

\begin{prop} \label{semipositivity properties}
Let $L$ be a line bundle on a paracompact strictly $K$-analytic space $V$. 
Let $x \in V$ and let $\metr$ be a piecewise $\Q$-linear  metric on $L$.

\begin{itemize}
 \item[(a)] The set of points in $V$ where $\metr$ is semipositive is open in $V$.
 \item[(b)] The tensor product of two piecewise $\Q$-linear  metrics which are semipositive at $x$ is again semipositive at $x$.
 \item[(c)] Let  $f:V'\to V$ be a morphism of paracompact strictly $K$-analytic spaces. 
If $\metr$ is semipositive at $x$, then
$f^*\metr$ is  semipositive at any point of $f^{-1}(x)$.
\end{itemize}
\end{prop}

\begin{proof}
Property (a) is obvious from the definitions. 
Property (b) follows easily from Proposition \ref{model test} and the linearity of the degree of a proper curve with respect to the divisor. Finally (c) follows from Proposition \ref{pullback of semipositive formal}.
\end{proof}


\begin{prop} \label{local vs global semipositivity}
Let $L$ be a line bundle on the {separated} paracompact strictly $K$-analytic space $V$ and let $\metr$ be a formal metric on $L$. Then $\metr$ is a semipositive formal metric as globally defined in \ref{definition semipositive formal metrics} if and only if $\metr$ is a semipositive piecewise linear  metric in every $x \in V$ as defined in \ref{local definition of semipositive formal metrics}. 
\end{prop}

\begin{proof}
The proof follows mainly the arguments in \cite[Proposition 6.4]{GK2}. 
It is clear that a semipositive formal metric is a semipositive piecewise linear metric in every $x\in X$ and we will prove now the converse.
Let $(\fV,\fL)$ be a formal $\kcirc$-model of $(V,L)$ with $\metr = \metr_\fL$. 
{Since the generic fiber $V$ is separated, it follows from \cite[Proposition 4.7]{bosch-luetkebohmert-1} that the model $\fV$ is also separated and hence we may apply \cite[Lemma 6.5.1]{chambert-loir-ducros}. This is a criterion which characterizes the points $v$  in the relative interior ${\rm Int}(V)$ over the base $K$ (defined in \cite[Definition 1.5.4]{berkovich-ihes}) by the property that the closure of the reduction of $v$ in the special fiber $\fV_s$ is proper over $\ktilde$.}

We assume that $\metr$ is semipositive at every $x \in V$. We choose a closed curve $C$ in $\fV_s$ which is proper over $\ktilde$. We have to show that $\deg_\fL(C) \geq 0$. 
By surjectivity of the reduction map $\pi:V \to \fV_s$, there is $x \in V$ such that $\pi(x)$ is the generic point of $C$. 
{Using \cite[Lemma 6.5.1]{chambert-loir-ducros}, the properness of  $C$ yields that $x \in {\rm Int}(V)$.}  
Since $\metr$ is semipositive at $x$, there is a compact strictly $K$-analytic neighborhood $W$ of $x$, a nef formal $\kcirc$-model $(\fW,\fM)$ of $(W,L|_W)$ {and a non-zero $k \in \N$} such that $\metr^{{\otimes k}} = \metr_\fM$ over $W$.  Using Proposition \ref{model test}, we may always replace the models $\fW$ and $\fV$ by dominating formal $\kcirc$-models and the line bundles $\fM$ and $\fL$ by their pull-backs. By \cite[Corollary 5.4 (b)]{bosch-luetkebohmert2}\footnote{{Note the misprint in \cite[Corollary 5.4 (b)]{bosch-luetkebohmert2}: \emph{immersion} should be replaced by \emph{open immersion}.}}, we may therefore assume that $\fW$ is a formal open subset of $\fV$. Then $\fL|_\fW^{{\otimes k}}$ is also a formal $\kcirc$-model of $L|_W^{{\otimes k}}$ and hence Proposition \ref{model test} implies that $\fL|_\fW$ is nef. 

Since $W$ is a neighborhood of $x$ and since {${\rm Int}(W)$ is contained in the intersection of ${\rm Int}(V)$ with the topological interior of $W$ in $V$ by \cite[Proposition 1.5.5]{berkovich-ihes}, we conclude that $x \in {\rm Int}(W)$.}
 Using \cite[Lemma 6.5.1]{chambert-loir-ducros} again, the closure of the reduction of $x$ in $\fW_s$ is proper over $\ktilde$ and hence equal to $C$. 
Since $\fL|_\fW$ is nef, it follows that $\deg_\fL(C) \geq 0$. 
\end{proof}

\begin{prop} \label{convexity and max for psh}
Let $\| \ \|_1$ and $\| \ \|_2$ be   algebraic metrics of the line bundle $L$ over the proper scheme $X$ over $K$.  
Then $\| \ \| := \min ( \| \ \|_1, \| \ \|_2)$ is an algebraic metric on $L$. If $\metr_1$ and $\metr_2$ are semipositive at $x \in \Xan$, then 
$\metr$ is semipositive at $x$.
\end{prop}

\begin{proof}
Since formal and algebraic metrics are the same as noted in Remark \ref{algebaic vs formal metric} 
and hence also the same as piecewise linear metrics, we deduce from Proposition \ref{properties for piecewise linear} (d)
  that $\| \ \|$ is an algebraic metric. If the given metrics are semipositive at $x$, then it remains to prove 
 that $\| \ \|$ is semipositive at $x$. 
By base change again, we may assume that $K$ is algebraically closed. By Lemma \ref{irreducible components}, 
we may assume that $X$ is a proper variety over $K$.
Let us pick models $\Lcal_1$, $\Lcal_2$ and  $\Lcal$ of $L$ defining the {algebraic} metrics $\| \ \| $,  $\| \ \|_1$ and $ \| \ \|_2$.
There is an algebraic $\kcirc$-model $\Xcal$ of $X$ on which $\Lcal_1$, $\Lcal_2$ and  $\Lcal$ are defined.  
There is an open neighborhood $W$ of $x$ in $\Xan$ such that $\metr_1$ and $\metr_2$ are semipositive at all points of $W$. We will show that $\metr$ is semipositive at every point of $W$. 
By \cite[6.5]{GK2}, it is equivalent to show that $\deg_\Lcal(C) \geq 0$ for any closed curve $C$ of $\Xcal_s$ contained in the reduction of $W$. Moreover, the same result yields that $\Lcal_1$ and $\Lcal_2$ restrict to nef line bundles on $C$.
By  
\cite[Theorem 4.1]{GK2}, there is a closed curve $Y$ in $X$ such that $C$ is an irreducible component of the special fibre of the closure $\overline{Y}$ in $\Xcal$. 
By restriction, we may assume that $X=Y$ is a curve and hence $C$ is an irreducible component of $\Xcal_s$. 
Let $\fX$ be the formal completion of $\Xcal$  and let $\fL, \fL_1,\fL_2$ be the line bundles on $\fX$ induced by the pull-backs of $\Lcal,\Lcal_1,\Lcal_2$. 

We have seen in the proof of Proposition \ref{model test} that we can associate to $\fX$ a canonical  formal $K^\circ$-model $\fX'$ of $\Xan$ with reduced special fibre and a canonical finite surjective morphism $\iota:\fX' \to \fX$.
So there is a closed curve $C'$ in $\fX'_s$ which maps onto $C$ in $\fX_s=\Xcal_s$. 
Let $\fL',\fL_1',\fL_2'$ be the line bundles on $\fX'$ given by pull-back of $\fL,\fL_1, \fL_2$. 
Note that $\fL',\fL_1',\fL_2'$ are formal $K^\circ$-models of the metrics $\metr,\metr_1,\metr_2$ on $\Lan$. 
By the projection formula, the line bundles $\fL_1',\fL_2'$ restrict to nef line bundles on $C'$ and it remains to show that 
\begin{equation} \label{positivity to show for deg}
\deg_{\fL'}(C') \geq 0.
\end{equation}
Let $\zeta$ be the generic point of $C'$.
Then there is a unique point $\xi$ in $\Xan$ with reduction $\zeta$.
This follows from  \cite[Proposition 2.4.4]{berkovich-book} since $\zeta$ has a formal affine open neighborhood in $\fX'$ of the form $\Spf(\Acal^\circ)$ for a strictly $K$-affinoid algebra $\Acal$.
Using $\| \ \| = \min ( \| \ \|_1, \| \ \|_2)$, we may assume $\metr(\xi)=\metr_1(\xi)$. 
Since $\Lan$ is algebraic, there is a non-trivial meromorphic section $t$ of $\fL'$. 
Note that the restriction of $t$ to the generic fibre $\Lan$ induces also a meromorphic section $t_1$ of $\fL_1'$. 
The meromorphic section $t/t_1$ of ${\mathfrak M}:=\fL' \otimes (\fL_1')^{-1}$ restricts to the trivial section $1$ of $\KO_\Xan$ and we have 
\begin{equation*} \label{metric ineq} \|t/t_1\|_{\mathfrak M} =\|t\|/\|t_1\|_1=\|t\|/\|t\|_1 \leq 1.
\end{equation*}  
By \cite[Proposition 7.5]{gubler98}, we deduce that $t/t_1$ is a global section of $\mathfrak M$. The definition of formal metrics and $\|t/t_1\|_{\mathfrak M}(\xi) = \|t\|(\xi)/\|t\|_1(\xi)=1$ yield that $\{y \in \Xan \mid \|t/t_1\|_{\mathfrak M}(y)\geq 1\}$ is the generic fibre of a formal open neighborhood $\fU$ of $\zeta$. Hence \cite[Proposition 7.5]{gubler98} again shows that $t/t_1$ is a nowhere vanishing regular section of $\mathfrak M$ on $\fU$. We conclude that the restriction of the global section $t/t_1$ to $C'$ is not identically zero inducing an effective Cartier divisor $D$ on $C'$. This shows 
$$\deg_{\mathfrak M}(C')= \deg_D(C') \geq 0.$$ 
Using that $\fL_1'$ is nef on $C'$ and $\fL'={\mathfrak M}\otimes \fL_1'$, we get 
$$\deg_{\fL'}(C') \geq \deg_{\fL_1'}(C') \geq 0$$
proving \eqref{positivity to show for deg}.
\end{proof}

\section{Plurisubharmonic model functions}
\label{section psh model functions}
We will introduce closed $(1,1)$-forms $\theta$ on a proper scheme $X$ over $K$ and $\theta$-psh model functions following the terminology in \cite{BFJ1}.

\begin{art}
\label{model metric}
Let $L$ be a line bundle on $X$.
We say that a metric $\| \ \|$ on $L^\an$ is a \emph{model metric} if there is a non-zero $d\in \N$ 
such that ${\| \ \|}^{\otimes d}$ is an algebraic metric on $(\Lan)^{\otimes d}$.
By Proposition \ref{formal vs piecewise linear metrics} and Remark \ref{algebaic vs formal metric}, 
$\| \ \|$ is a model metric if and only if it is a piecewise $\Q$-linear metric.
\end{art}

\begin{art}
\label{definition model functions}

We say that a function $\varphi : X^\an \to \R$ is a \emph{model function}  
if there exists $d\in \N_{>0}$ and $\| \ \|$ an algebraic metric on $\mathcal{O}_{X^\an}$ such that 
$\varphi = -\frac{1}{d} \log \| 1 \|$.
If we can take $d=1$, we say that $\varphi$ is a \emph{$\Z$-model function}.
The set of model functions on $X$ is denoted  by $\mathcal{D}(X)$. 
\end{art}

\begin{art} \label{lemma model function Cartier}
Let $\Xcal$ be an algebraic $\Ko$-model of $X$. 
A vertical Cartier divisor on $\Xcal$ is a Cartier divisor $D$  on $\Xcal$ which is supported on the special fiber $\Xcal_s$.
A vertical Cartier divisor $D$ on $\Xcal$ determines a model $\Ocal(D)$ of $\Ocal_X$  hence an associated model function 
\begin{equation*} 
\varphi_D \coloneqq -\log \| 1 \|_{\Ocal(D)} : X^\an \to \R
\end{equation*}
Note that every $\Z$-model function has this form. 
Indeed, if $\Lcal$ is an algebraic $K^\circ$-model of $\Ocal_X$ with $\varphi = - \log \metr_\Lcal$, then the section $1$ of $\Ocal_X$ extends to a meromorphic section $s$ of $\Lcal$ and the vertical Cartier divisor $D := \div(s)$ satisfies $\varphi=\varphi_D$.

{For example, the constant functions on $\Xan$ with values in the value group $\Gamma = - \log |K^\times|$ are 
the $\Z$-model functions of the form $\varphi_D$ with $D= \div(\alpha)$ for non-zero elements $\alpha \in K$. Moreover, the constant functions on $\Xan$ with values in $\Q\Gamma$ are $\Q$-model functions.}
\end{art}

\begin{art}
We set $\Pic(\Xcal)_\R \coloneqq \Pic(\Xcal)\otimes_\Z \R$.
We define the \emph{N\'{e}ron-Severi} group as the $\R$-vector space $\Pic(\Xcal)_\R$ modulo the subspace generated by numerically trivial line bundles. 
We   denote this space by $N^1(\Xcal / S)$, where $S := \Spec(\kcirc)$.
The space of \emph{closed $(1,1)$-forms} on $X$ is defined as the direct limit 
\[ \mathcal{Z}^{1,1}(X) \coloneqq \varinjlim N^1(\Xcal / S)\]
where the limit is taken over all algebraic $\kcirc$-models of $X$.
We say that a closed $(1,1)$-form $\theta$ is \emph{determined on some model $\Xcal$} if it is in the image of the map 
$N^1(\Xcal / S) \to  \mathcal{Z}^{1,1}(X)$.
\end{art}

\begin{art}
\label{curvature}
We denote by $\Pichat(X)$ the group of isomorphism classes of line bundles on $X$ equipped with a model metric. 
There is a  map $c_1 : \Pichat(X) \to \mathcal{Z}^{1,1}(X)$ which sends the class of 
$(L,\| \ \|_\KL)$ to the class of $\KL$. 
We will show below that this map is well defined. 
We denote its image by $c_1(L,\| \ \|_\KL)$ and call it the \emph{curvature form} of $(L, \| \ \|_\KL)$. 
We get a natural linear map $dd^c: \mathcal{D}(X) \to \mathcal{Z}^{1,1}(X)$ which maps a $\Z$-model function $\varphi$ to $c_1(\Ocal_\Xan,\| \ \|_\varphi)$, where $\metr_\varphi$ is the corresponding algebraic metric on $\Ocal_\Xan$\footnote{{For a $dd^c$-lemma, see \cite[Theorem 4.3]{BFJ1} in the discretely valued case and \cite[Theorem 4.2.7]{Jell-thesis} for a generalization to non-discrete valuations.}}.

To show that the curvature $c_1(L,\| \ \|_\KL)$ is well defined in $\mathcal{Z}^{1,1}(X)$, we consider algebraic $K^\circ$-models $\Lcal,\Lcal'$ of $L$ with $\metr_\Lcal=\metr_{\Lcal'}$. 
We have to show that $\Lcal$ and $\Lcal'$ are numerically equivalent on a suitable algebraic $\kcirc$-model of $X$. 
Since the isomorphism classes of algebraic $\kcirc$-models of $X$ form a directed set, we may assume that $\Lcal$ and $\Lcal'$ are line bundles on the same algebraic $\kcirc$-model $\Xcal$. 
Using base change to $\C_K$, Lemma \ref{lemma semipositive base change} shows that we may assume $K$ algebraically closed. By Lemma \ref{semipositive reduced}, we may assume that $\Xcal$ is reduced. Then  the formal completion $\fX$  of $\Xcal$ is also reduced. We consider the canonical formal $\kcirc$-model $\fX'$ with reduced special fibre and with a  finite surjective morphism $\fX' \to \fX$ extending $\id_{\Xan}$ as in the proof of Proposition \ref{model test}.
We may apply \cite[Proposition 7.5]{gubler98} to the line bundles $\fL,\fL'$ on $\fX'$ induced by $\Lcal,\Lcal'$. Since $\metr_\fL=\metr_\Lcal=\metr_{\Lcal'}=\metr_{\fL'}$, we deduce that $\fL \cong \fL'$. Using that $\fX_s=\Xcal_s$, the projection formula applied to the finite surjective morphism $\fX_s'\to \fX_s$ shows that $\Lcal$ is numerically equivalent to $\Lcal'$.
\end{art}

\begin{art} \label{positive forms}
We say that an element of $N^1(\Xcal /S)$ is \emph{ample} if it is the class of a non-empty sum $\sum_i a_i c_1(\Lcal_i)$ 
for some real numbers $a_i >0$ and some ample line bundles $\Lcal_i$. 
For an algebraic $K^\circ$-model $\Xcal$ of $X$,  a closed $(1,1)$-form $\theta$ is called
{\it $\Xcal$-positive} if 
{it is determined on $\Xcal$ by some 
$\theta_\Xcal \in N^1(\Xcal /S)$  which is ample.}
We say that a model metric $\metr$ of  a line bundle $L$ is 
{\it $\Xcal$-positive} if the same holds for the curvature form $c_1(L,\metr)$.
We say that an element $\theta \in N^1(\Xcal / S)$ is nef if $\theta \cdot C\geq 0$ for any closed curve $C \subset \Xcal_s$. 
A closed $(1,1)$-form $\theta$ is said to be \emph{semipositive} if 
it is determined by a nef class $\theta_\Xcal \in N^1(\Xcal /S)$ on a model $\Xcal$.

If $\theta$ is a closed $(1,1)$-form, we say that a model function $\varphi$ is 
\emph{$\theta$-plurisubharmonic} (briefly \emph{$\theta$-psh}) if 
$\theta +dd^c\varphi$ is semipositive. 
If $\theta$ is the closed $(1,1)$-form associated with some line bundle $\Lcal$ on $\Xcal$  
and if $D$ is a vertical Cartier divisor on $\Xcal$, then by definition $\varphi_D$ is a $\theta$-psh function if and only if 
$\Lcal \otimes \Ocal(D)$ is nef if and only if $\| \ \|_{\Lcal \otimes \Ocal(D)}$ is a semipositive metric.
\end{art}

\begin{art}
\label{semipositive metric vs psh}
Let $L$ be a line bundle on $X$.
Let $\| \ \|$ be a model metric on $L^\an$ and $\theta \coloneqq c_1(L, \| \ \|)$.
Let $\| \ \|'$ be another metric on $L^\an$ and let $\varphi \coloneqq -\log( \| \ \|' / \| \ \|)$.
Then $\| \ \|'$ is a model metric if and only if $\varphi$ is a model function. 
Moreover $\| \ \|'$ is a semipositive model metric if and only if $\varphi$ is a $\theta$-psh model function. 
\end{art}

\begin{art} \label{vertical nef and generic fibre}
The  N\'eron--Severi group $N^1(X)$ of $X$ is the group $\Pic(X)\otimes_\Z \R$ modulo the subspace generated by the numerically trivial line bundles.
For a closed $(1,1)$-form $\theta$, let $\{\theta\}$ be the associated de Rham class, given by $\{\theta\}=\theta_\Xcal|_X \in N^1(X)$ for any algebraic $\kcirc$-model $\Xcal$  on which $\theta$ is determined by $\theta_\Xcal \in N^1(\Xcal / S)$. 
If $\theta$ is semipositive, then $\{\theta\}$ is nef. 

To see this, we choose any closed curve $C$ in $X$ and non-zero $\rho$ in the maximal ideal of the valuation ring $\kcirc$. 
Then using the divisorial intersection theory  in \cite{gubler98}, we have
$$v(\rho)\deg_{\{\theta\}}(C) =  \deg(\Div(\rho) . \theta_\Xcal . \overline{C}) = \deg(\theta_\Xcal . \Div(\rho) . \overline{C}) = v(\rho)\deg_{\theta_\Xcal}(\overline{C}_s).$$
Since $\theta_\Xcal$ is 
nef, the degree of the special fibre $\overline{C}_s$ of the closure $\overline{C}$  in $\Xcal$ 
is non-negative proving the claim.
\end{art}

\begin{lem} \label{model functions and multiplicities}
Let us assume that $K$ is algebraically closed.
Let $\varphi$ be a  model function determined by a vertical Cartier divisor $D$ on the algebraic $\kcirc$-model $\Xcal$ of $X$. We assume that the special fibre $\Xcal_s$ is reduced. 
Then $\varphi \geq 0$ if and only if the Cartier divisor $D$ is effective. 
\end{lem}

\begin{proof}
By definition of a Cartier divisor, $\mathcal{O}_{\Xcal} (-D)$ is a coherent subsheaf of the sheaf of meromorphic functions on $\Xcal$, 
and $D$ is effective if and only if $\mathcal{O}_{\Xcal} (-D)$ happens to be a subsheaf of $\mathcal{O}_{\Xcal}$. 
Since $D$ is a vertical Cartier divisor, there exists $a \in K^\circ \setminus \{0\}$ such that $D+\div(a)$  is effective, i.e. 
$a\mathcal{O}_{\Xcal} (-D)$ is a subsheaf of $\mathcal{O}_{\Xcal}$.
Now   $D$ is effective if and only if $a\mathcal{O}_{\Xcal} (-D)$ is a subsheaf of $a\mathcal{O}_{\Xcal}$ when they are both considered as coherent subsheaves of $\mathcal{O}_{\Xcal}$.
Since the completion functor is fully faithful on the category of coherent sheaves by \cite[Proposition I.9.4.2]{fujiwara-kato-1}, 
it is equivalent to check the associated inclusion on the formal completion of $\hat{\Xcal}$.

We now get the claim from the corresponding statement for admissible formal schemes which is proven in \cite[Proposition  A.7]{gubler-rabinoff-werner} and hence applies to the formal completion $\hat{\Xcal}$ of $\Xcal$ and its Cartier divisor $\hat{D}$ given by pull-back of $D$. 
\end{proof}

\begin{rem} \label{model functions and multiplicities for normal}
If we assume that $\Xcal$ is normal instead of assuming that $\Xcal_s$ is reduced, then Lemma \ref{model functions and multiplicities}  holds for any non-archimedean field $K$ (see \cite[Corollary 2.12]{GS15}). 
\end{rem}

We recall the following result from  \cite[Corollary 1.5]{BFJ1}. For convenience of the reader and to check that no noetherian hypotheses are used, we give here a proof. 

\begin{prop} \label{extension of ampleness}
Let $L$ be an ample line bundle on the projective scheme $X$ over $K$ 
and let $\Xcal_0$ be any algebraic $\kcirc$-model of $X$.  
Then there is an algebraic $\kcirc$-model $\Xcal$ of $X$ dominating $\Xcal_0$ and an ample line bundle $\Lcal$ on $\Xcal$ which is a $\kcirc$-model of $L^{\otimes m}$ for a suitable $m \in \N$.
\end{prop}

\begin{proof}
Every algebraic $\kcirc$-model of a projective scheme  $X$ is dominated by a projective algebraic $\kcirc$-model \cite[Proposition 10.5]{gubler03}.
Hence we may assume that $\Xcal_0$ is projective. There is $m_1 \in \N$ and a closed immersion of $X$ into $\P_K^N$ such that $L^{\otimes m_1} = \Ocal_{\P_K^N}(1)|_X$. 
Then the schematic closure of $X$ in $\P_\kcirc^N$ is an algebraic $\kcirc$-model $\Xcal_1$ of $X$.
Moreover, $\Xcal_1$ 
has an ample line bundle $\Lcal_1$ such that $\Lcal_1|_X=L^{\otimes m_1}$. 
Then the schematic closure of the diagonal in $\Xcal_0 \times_\kcirc \Xcal_1$ is a projective $\kcirc$-model $\Xcal$ of $X$.  Now the claim follows from the following lemma applied to $f:=p_1$. 
\end{proof} 

\begin{lem} \label{precise extension of ampleness}
Let $f:\Xcal \to \Xcal_1$ be a morphism of projective algebraic $\kcirc$-models of $X$ extending $\id_X$ and assume that $L^{\otimes m_1}$ extends to an ample line bundle $\Lcal_1$ on $\Xcal_1$ for some  $m_1 \in \N$. 
Then there is a positive multiple $m$ of $m_1$ such that $L^{\otimes m}$ extends to an ample line bundle on $\Xcal$. 
\end{lem}

\begin{proof}
Note that $f$ is a projective morphism, hence there is a closed immersion of $\Xcal$ into a projective space $\P_{\Xcal_1}^k$ over $\Xcal_1$. Let $\Ecal$ be the restriction of $\Ocal_{\P_{\Xcal_1}^k}(1)$ to $\Xcal$. Since $\Ecal$ is relatively ample with respect to $f$ and since  $\Lcal_1$ is an ample line bundle on $\Xcal_1$, there is $m_2 \in \N$ such that $\Lcal :=f^*(\Lcal_1)^{\otimes m_2} \otimes \Ecal$ is ample on $\Xcal$. Then $\Lcal$ is an ample $\kcirc$-model of $L^{\otimes m}$ for $m:=m_1m_2$.
\end{proof}

\begin{prop} \label{perturbation of positivity}
Let $\omega$ be  $\Xcal$-positive and let $\theta$ be any closed $(1,1)$-form determined on $\Xcal$. 
Then $\omega + \ve \theta$ is $\Xcal$-positive for $\ve \in \R$ sufficiently close to $0$.
\end{prop}

\begin{proof}
Since $\Spec(\kcirc)$ is affine, ample  is the same as relatively ample. It remains to check that the restriction of $\omega + \ve \theta$ to the special fibre is ample (see \cite[9.6.4 and 9.6.5]{EGAIV3}). 
The ample cone on the special fiber is the  interior of the nef cone. This proves immediately the claim.  
\end{proof}

The following result generalizes \cite[Proposition 5.2]{BFJ1}. Again, we follow their arguments.

\begin{prop} \label{semipositive vs ample}
Let $\theta$ be a closed $(1,1)$-form with ample de Rham class $\{\theta\} \in N^1(X)$ and let $\Xcal_0$ be any algebraic $\kcirc$-model of $X$. 
Then there is an algebraic $\kcirc$-model $\Xcal$ of $X$ dominating $\Xcal_0$ such that $\theta$ is determined on $\Xcal$ and a  model function $\varphi$ such that $\theta + dd^c\varphi$ is $\Xcal$-positive. If $\theta$ is semipositive and $\ve >0$, then we may find such a model function with $-\ve \leq \varphi \leq 0$.
\end{prop}

\begin{proof}
For any algebraic $\kcirc$-model $\Xcal$, the canonical homomorphism $N^1(\Xcal/S) \to N^1(\Xcal_s)$ is injective by definition of numerical equivalence. 
The ample cone of $N^1(\Xcal/S)$ is the preimage of the ample cone of $N^1(\Xcal_s)$ (see the proof of Proposition \ref{perturbation of positivity}).

By assumption, $\theta$ can be represented by a finite sum $\sum_i \lambda_i c_1(\Lcal_i)$ with line bundles $\Lcal_i$ on algebraic $\kcirc$-models $\Xcal_i$ of $X$ and $\lambda_i \in \R$. 
{We set $L_i \coloneqq \Lcal_i|_{X}$.
Hence $\sum_i \lambda_i c_1(L_i)$ represents $\{ \theta\}$.}
Since $\{\theta\}$ is ample, there are finitely many ample line bundles $H_j$ on $X$ 
and $\mu_j>0$ such that $\sum_j \mu_j c_1(H_j)$  {also} represents $\{\theta\}$. 
Clearly, we may assume that every $H_j$ is very ample and hence $H_j$ has a very ample $\kcirc$-model $\Hcal_j$ on a projective algebraic $\kcirc$-model $\Ycal_j$ of $X$. 
Let $\vartheta  \in \mathcal{Z}^{1,1}(X)$ be the class of $\sum_j \mu_j c_1(\Hcal_j)$.
We recall that the isomorphism classes of $\kcirc$-models of $X$ form a directed set and that any $\kcirc$-model of the projective variety $X$ is dominated by a projective $\kcirc$-model. We conclude that there is a projective $\kcirc$-model $\Xcal$ of $X$ such that $\id_X$ extends to morphisms $\Xcal \to \Xcal_i$ and $\Xcal \to \Ycal_j$ for every $i$ and $j$. Replacing $\Lcal_i$ by its pull-back to $\Xcal$, we  may assume that $\Xcal_i = \Xcal$ meaning that $\Lcal_i$ lives on $\Xcal$ for every $i$. 
Moreover, it follows from Lemma \ref{precise extension of ampleness}  that $H_j$ extends to an ample line bundle on $\Xcal$ and hence we may assume that every $\Hcal_j$ lives on $\Xcal$ as well. 
This means that  $\vartheta$ is $\Xcal$-positive.

We now consider the linear equation in the variables 
$(\lambda_i')$, $(\mu'_j)$
\begin{equation}
\label{eq linear system}
\sum_i \lambda_i' c_1(L_i)  - \sum_j \mu_j' c_1(H_j) =0
\end{equation}
which we consider  as an equation  in $\Pic(X) \otimes_\Z \Q \subset \Pic(X) \otimes_\Z \R$.
By assumption, $(\lambda_i)$, $(\mu_j)$ is a real solution of \eqref{eq linear system}.
Hence, by a linear algebra argument, for any $\delta >0$, we can find a solution $(\lambda_i')$, $(\mu'_j)$ of \eqref{eq linear system} 
with $\lambda'_i, \mu_j' \in \Q$ such that 
$|\lambda'_i - \lambda_i| <\delta$ and $|\mu_j' -\mu_j|<\delta$.
We can take $\delta$  small enough to get $\mu_j' >0$. 
We may assume that the $\Q$-line bundles 
$$\Lcal':=\bigotimes_i  \Lcal_i^{\otimes \lambda_i'} \quad {\rm and} \quad \Hcal':=\bigotimes_j  \Hcal_j^{\otimes \mu_j'}$$
agree on the generic fibre $X$. 
Let $\varphi$ be the model function corresponding to $\Hcal' \otimes (\Lcal')^{-1}$. 
As all $\mu_j'$ are positive and all $\Hcal_j$ are ample on $\Xcal$,  $\Hcal'$ is an ample $\Q$-line bundle on $\Xcal$.
By definition we have 
\[ 
\theta + dd^c \varphi = 
\theta + c_1(\Hcal') - c_1(\Lcal')  =
( \theta - c_1(\Lcal') ) + ( c_1(\Hcal') - \vartheta) + \vartheta.
\]
Taking $\delta$ small enough, we  can make the first two terms 
$( \theta - c_1(\Lcal') )$ and $( c_1(\Hcal') - \vartheta)$ as small as we want in $N^1(\Xcal/S)$.
In addition, we know that $\vartheta$ is ample on $\Xcal$.
It follows from our remark at the beginning that the ample cone is open in $N^1(\Xcal/S)$.  For $\delta$ small enough, we deduce that 
$\theta + dd^c \varphi$ is ample on $\Xcal$.

Now let us assume that $\theta$ is semipositive. 
Since a  function in $\mathcal{D}(X)$ is continuous on $\Xan$, it is bounded.
We may replace $\varphi$ by $\varphi - c$ for any sufficiently large $c$ in the value group $\Gamma$ without changing $dd^c\varphi$
and hence we may assume $\varphi \leq 0$. 
Since the sum of a nef and an ample class in $N^1(\Xcal / S)$ remains ample 
(as we can check that on the special fibre, see the remark at the beginning of the proof), we know that 
$$\theta + dd^c(\ve\varphi)= \ve(\theta +dd^c\varphi) + (1-\ve)\theta$$
is also $\Xcal$-positive  for all $0 < \ve \leq 1$. 
Using  a rational $\ve>0$  sufficiently close to $0$, we  get the last claim for the model function $\ve \varphi$.
\end{proof}

\begin{art}
\label{remark Kiehl}
We want to recall a result due to Kiehl \cite[Theorem 2.9]{Kie72}  that we will use later.
Let $\Ycal$ be a  $K^\circ$-scheme. 
Let  $f \colon \Xcal \to \Ycal$ be a proper morphism of finite presentation and let $\Mcal$ be a coherent $\mathcal{O}_\Xcal$-module.
Then $f_*(\Mcal)$ is a coherent $\mathcal{O}_\Ycal$-module.
For some explanation on why Kiehl's result implies this, we refer to Example 3.3 and 3.5 in \cite{Ulr95}. 

We will apply this result in the case of proper flat schemes $\Xcal,\Ycal$  over $\kcirc$. By [Raynaud-Gruson, Corollaire 3.4.7], they are finitely presented over $\kcirc$ and hence any  $f \colon \Xcal \to \Ycal$ is proper and of finite presentation.
\end{art}

\begin{art}
\label{Stein factorization}
We also recall a non-noetherian version of the Stein factorization theorem that will be used later. 
We quote the following result from \cite[Tag 03GY]{stacks-project}.

Let $f : X \to S$ be a universally closed and
quasi-separated morphism of schemes. Then there exists a factorization
$$
\xymatrix{
X \ar[rr]_{f'} \ar[rd]_f & & S' \ar[dl]^{g} \\
& S &
}
$$
with the following properties:
\begin{enumerate}
\item the morphism $f'$ is universally closed, quasi-compact, quasi-separated 
and surjective;
\item the morphism $g : S' \to S$ is integral;
\item  $f'_*\mathcal{O}_X = \mathcal{O}_{S'}$;
\item The relative spectrum of $f_*\mathcal{O}_X$ over $S$ is equal to $S'$;
\end{enumerate}
\end{art}

In the following, we consider an admissible formal scheme $\fX$ over $\kcirc$. 
A {\it  vertical coherent fractional ideal } $\mathfrak{a}$ on $\fX$ is  an $\Ocal_{\fX}$-submodule of $\Ocal_{\fX}\otimes_{K^\circ} K$  with generic fibre $\mathfrak{a}_\eta = \Ocal_{\fX_\eta}$ such that for every formal open affine subset $\fU$, there is a non-zero $\alpha \in \kcirc$ with $\alpha \mathfrak{a}|_\fU$ a coherent ideal sheaf on $\fU$.

\begin{definition}
\label{definition function ideal formal}
Let $\mathfrak{a}$ be a vertical coherent fractional ideal on $\fX$.
We define the function 
\begin{equation}
\begin{array}{cccc}
   \log |\mathfrak{a}| \colon  &\fX_\eta     & \to         &        \R  \\
    &  x    & \mapsto &   \max \{ \log |f(x)| \ \big| \ f\in \mathfrak{a}_{\pi(x)} \}
\end{array}
\end{equation}
where $\pi:\fX_\eta \to \fX_s$ is the reduction map. 
This is a continuous map because $\mathfrak{a}$ is  a coherent sheaf.
\end{definition}

We will now use divisorial points as introduced and studied in Appendix \ref{appendix on divisorial points}.

\begin{lem}
\label{lemma convexity}
Let $\fX$ be an admissible formal scheme over $\kcirc$ and let $f\in \mathcal{O}(\fX_\eta)$.
Let $I$ be the set of divisorial points associated with $\fX$.
Then 
\[ \sup_{x\in \fX_\eta} |f(x)| = \sup_{x\in I} |f(x)|.\]
\end{lem}

\begin{proof}
Since we can work locally on $\fX$, we can assume that $\fX=\Spf(A)$ is formal affine, in which case the supremum is not $+\infty$. 
It follows from Proposition \ref{divisorial points associated to formal affine model} that the set of divisorial points $I$ is the Shilov boundary of $\fX_\eta$ proving precisely our claim.
\end{proof}

\begin{cor}
\label{corollary convexity}
Let $\fX$ be an admissible formal scheme over $\kcirc$ and let 
$\varphi \coloneqq \log|\mathfrak{a}|$ for  a vertical coherent fractional ideal $\mathfrak{a}$ on $\fX$.
Let $I$ be the set of divisorial points associated with $\fX$.
Then 
\[ \sup_{x\in \fX_\eta} \varphi(x)= \sup_{x\in I} \varphi(x).\]
\end{cor}

\begin{proof}
We can easily replace $\fX$ by one of its open formal subschemes 
because divisorial points are compatible with formal open subsets.
So we can assume that $\fX$ is an admissible  formal affine scheme.
Hence  $\mathfrak{a}$ is generated by finitely many functions $f_1, \ldots, f_n$. 
Then $\varphi(x) = \max_{j=1\ldots n} \log |f_j(x)|$. 
The result follows from Lemma \ref{lemma convexity}.
\end{proof}

Let us return to our algebraic setting with an algebraic $\kcirc$-model $\Xcal$ of the proper scheme $X$.  
We will apply the above to the formal completion $\fX'$ of the algebraic $\kcirc$-model $\Xcal'$ from the following result.

\begin{lem}
\label{lemma limit psh function}
Let $\theta$ be a closed $(1,1)$-form determined on $\Xcal$ and let $\varphi$ be a $\theta$-psh model function on $X$.
Then there is an algebraic $\kcirc$-model $\Xcal'$ of $X$ with a finite morphism $\Xcal' \to \Xcal$ extending $\id_X$ and  a sequence $\mathfrak{a}_m$ of vertical coherent fractional ideals on $\Xcal'$ such that 
$\frac{1}{m} \log | \mathfrak{a}_m|$ converges uniformly to $\varphi$.
\end{lem}

\begin{proof}
The proof of \cite[5.7]{BFJ1} can be adapted in our non-noetherian context. 
For the convenience of the reader we detail this.

\emph{Step 1.} We have seen in \ref{lemma model function Cartier} that there is an algebraic $\kcirc$-model $\Ycal$ and a vertical  Cartier divisor $D$ on $\Ycal$ with $\varphi = \varphi_D$. 
By \ref{algebraic models}, we may assume that the identity $\id_X$ extends to a morphism $\pi:\Ycal \to \Xcal$. 
It follows from Raynaud's theorem that there is an admissible formal blowing up $\psi:\fY \to \hat{\Xcal}$ of the formal completion $\hat{\Xcal}$ in an open coherent ideal ${\mathfrak b}'$ such that $\fY$ dominates the formal completion of $\Ycal$. 
By the formal GAGA-principle for proper schemes over $\kcirc$ proved by Fujiwara--Kato \cite[Theorem I.10.1.2]{fujiwara-kato-1}, the coherent ideal $\mathfrak b'$ is the formal completion of a coherent vertical ideal $\mathfrak b$ on $\Xcal$ and hence $\psi$ is the formal completion of the blowing up of $\Xcal$ in $\mathfrak b$. Hence we may assume that $\pi$ is precisely this algebraic blow up morphism.

\emph{Step 2.} 
Note that $\pi$ is a proper morphism and hence $\pi_*(\Ocal_\Ycal)$ is a coherent sheaf by \ref{remark Kiehl}. Let $\pi= g \circ \pi'$ be the Stein factorization of $\pi$ as in \ref{Stein factorization}. It follows from coherence of $\pi_*(\Ocal_\Ycal)$ and from \ref{Stein factorization}(4) that the morphism $g:\Xcal' \to \Xcal$ is finite. By construction, $\Xcal'$ is a model of $X$ and $\pi',g$ restrict to the identity on $X$.

\emph{Step 3.} 
Let $C \subset \Ycal_s$ be a curve which is contracted by $\pi$ i.e. such that  $\pi (C) = \{x\}$ for some closed point $x\in \Xcal_s$.
Since $\varphi$ is $\theta$-psh, by definition we get that 
$(D + \pi^*(\theta))\cdot C \geq 0$.
On the other hand, by the projection formula, 
$\pi_*(\pi^*(\theta) . C )= \theta . \pi_*(C)=0$.
Hence $\pi^*(\theta)\cdot C=0$ and so 
 $D\cdot C \geq 0$. 
By definition, this means that $D$ is $\pi$-nef.

\emph{Step 4.}
By the construction in Step 1, there is a vertical ideal sheaf $\mathfrak{b}$ on $\Xcal$ such that $\pi$ is the blow up of $\Xcal$ along $\mathfrak{b}$.
By the universal property of the blow up, $\mathfrak{b} \mathcal{O}_{\Ycal} = \mathcal{O}_{\Ycal}(H)$ for 
an effective $\pi$-ample vertical Cartier divisor $H$ on $\Ycal$.

\emph{Step 5.} 
We choose a non-zero $k \in \N$. Let $x\in \Xcal_s$ be a closed point. 
{We denote by $\Ycal_x$ the fiber over $x$ with respect to the morphism $\Ycal_s \to \Xcal_s$. Note that $\Ycal_x$ is a proper scheme over the residue field of $x$.} 
Since  $H$ is $\pi$-ample, $\mathcal{O}(H)_{|\Ycal_x}$ is ample.
Similarly, since $D$ is $\pi$-nef, $\mathcal{O}(D)_{|\Ycal_x}$ is nef.
It follows from Kleiman's criterion that
$\mathcal{O}(kD+H)_{|\Ycal_x}$ is ample.
Hence by \cite[Corollaire 9.6.5]{EGAIV3}, $kD+H$ is $\pi$-ample  
and hence $kD+H$ is also $\pi'$-ample.

\emph{Step 6.}
By \ref{Stein factorization}(3), we have $\pi'_* \mathcal{O}_{\KY} = \mathcal{O}_{\KX'}$.
It follows that $\pi'_*$ maps vertical coherent fractional ideals on $\Ycal$ to vertical coherent fractional ideals on $\Xcal'$.
 It follows   that 
$\mathfrak{a} \coloneqq \pi'_* \mathcal{O}_{\Ycal}(n(kD+H)+lD)$ is a  vertical coherent fractional ideal on $\Xcal'$ for every $n \in \N$ and $l=0,\dots,k-1$.

\emph{Step 7.} 
Hence by the characterization given in \cite[Proposition 4.6.8]{EGAII}  of $\pi'$-ampleness,  for all sufficiently large $n\in \N$, the map
$\pi'^* \pi'_* \mathcal{O}_{\Ycal}(n({k}D+H)+lD) \to \mathcal{O}_{\Ycal}(n({k}D+H)+lD)$ is surjective which means that $\pi'^* \mathfrak{a} \to \Ocal_{\Ycal}(n(kD+H)+lD)$ is surjective.
This implies that 
$\log |\mathfrak{a}| = \varphi_{n(kD+H)+lD}$ and hence
$$\frac{1}{m}\log |\mathfrak{a}| = \varphi_D + \frac{n}{m}\varphi_H$$
for $m \coloneqq nk + l$. We have $0 \leq \frac{n}{m}\varphi_H \leq \frac{1}{k} \varphi_H$ and this is arbitrarily small for sufficiently large $k$ independently of the choice of $n$ and $l$. This leads easily to the construction of an approximating sequence as in the claim.
\end{proof}

\begin{rem} \label{approximation and normal}
If $\Xcal$ is normal, then we have $\Xcal=\Xcal'$ in Lemma \ref{lemma limit psh function}.
\end{rem}

\begin{prop} \label{theta-psh and skeleton}
Let $\Xcal$ be an algebraic $\kcirc$-model of the proper scheme $X$ over $K$. 
Let $I$ be the set of divisorial points of $\Xan$ associated with $\Xcal$.
Let $\theta$ be a  closed $(1,1)$-form which is determined on $\Xcal$ and let $\varphi$ 
be a $\theta$-psh model function on X. 
Then 
\[ \sup_{x\in \Xan} \varphi(x) = \sup_{x\in I} \varphi(x).\]
\end{prop}

\begin{proof}
Let $\Xcal'$ be the model of $X$ from Lemma \ref{lemma limit psh function}. 
Since the constructed morphism $\Xcal'\to \Xcal$ is finite, the set of divisorial points of $\Xan$ associated with $\Xcal'$ agrees with $I$.
Then the claim follows from Lemma \ref{lemma limit psh function} and from Corollary \ref{corollary convexity} applied to the formal completion $\fX'$ of $\Xcal'$.
\end{proof}

\section{Semipositivity and pointwise convergence}
\label{section semipositivity and pointwise convergence}

Our goal is to generalize \cite[Theorem 5.11]{BFJ1}  to a line bundle $L$ 
on a proper {scheme} $X$ over any non-archimedean field $K$.
This is a generalization in various aspects as in \cite{BFJ1}, $X$ was assumed to be a smooth projective variety and the valuation was discrete with residue characteristic zero (due to a use of the theory of multiplier ideals). 

In terms of metrics, the main result means that pointwise convergence of semipositive  model metrics on $L^\an$ to a model metric  implies that the limit is a semipositive model metric. 
By Chow's lemma, we will reduce to the case of projective varieties. 

We will first prove an analogue of \cite[Lemma 5.12]{BFJ1}.
Recall that we denote by $X^{\rm div}$ the set of divisorial points 
of the analytification $\Xan$ (see Appendix \ref{appendix on divisorial points}).

\begin{prop} \label{Lemma 5.12}
Let $X$ be a projective scheme over $K$ with an ample line bundle $L$. 
We consider an algebraic $\kcirc$-model $\Xcal$ of $X$ and a line bundle $\Lcal$ on $\Xcal$ extending $L$. 
Let $\metr=\metr_\Lcal$ be the corresponding model metric on $L^\an$ which is assumed to be the pointwise limit over $X^{\rm div}$ of semipositive model metrics on $L^\an$.  
Then $\metr$ is a semipositive model metric.
\end{prop}

By  Lemma \ref{lemma semipositive base change} and Lemma \ref{irreducible components}, it is enough to check the claim for a projective variety over an algebraically closed field $K$. 
Here we have used  that $(X \otimes \C_K)^{\rm div}$ is the preimage of $X^{\rm div}$ with respect to base change morphism $(X \otimes \C_K)^{\rm an} \to \Xan$ 
(see Proposition \ref{divisorial points and base change}), and that 
$X^{\rm div}= \bigcup (X_i)^{\rm div}$
where $X_i$ ranges over the irreducible components of $X$ (see Proposition \ref{divisorial points and irreducible components})
Then Proposition \ref{Lemma 5.12} follows immediately from Lemma \ref{Lemma 5.12-first step} and Lemma \ref{Lemma 5.12-second step} below. 

Recall that the base-ideal $\mathfrak{a}_m$ of $\Lcal^{\otimes m}$
is defined as the image of the canonical map
$$H^0(\Xcal, \Lcal^{\otimes m}) \otimes \Lcal^{\otimes (-m)} \rightarrow \Ocal_\Xcal.$$
Since $L$ is ample on $X$, $\mathfrak{a}_m$ is a vertical coherent ideal sheaf for $m$ sufficiently large. 

We give now the analogue of Definition \ref{definition function ideal formal} in the algebraic setting:

\begin{definition}
\label{definition function ideal}
Let $\mathfrak{a}$ be a coherent fractional ideal 
on the algebraic  $\kcirc$-model $\Xcal$ of the proper scheme $X$ over $K$.
Then we set
$$ \log |\mathfrak{a}|(x) \coloneqq \max \{ \log |f(x)| \ \big| \ f\in \mathfrak{a}_{\pi(x)} \} \in [-\infty, \infty[$$
where $\pi:\Xan \to \Xcal_s$ is the reduction map. 
\end{definition}

\begin{lem}
\label{Lemma 5.12-first step}
We keep the same hypotheses as in Proposition \ref{Lemma 5.12}. We assume additionally that $K$ is algebraically closed and that $X$ is a variety. We fix a finite subset $S$ of $X^{\rm div}$.    
Then there is a sequence of algebraic $\kcirc$-models $\Zcal_m$ such that $\id_X$ extends to finite morphisms $g_m:\Zcal_m \to \Xcal$ with the following property: Let $\mathfrak b_m$ be the base-ideal of $g_m^*(\Lcal)^{\otimes m}$. For $m$ sufficiently large, $\mathfrak b_m$ is a vertical coherent ideal on $\Zcal_m$ and $\frac{1}{m}{\log |\mathfrak b_m|}$ converges pointwise to $0$ on $S$.
\end{lem}

This lemma is similar to the first step in the proof of \cite[Lemma 5.12]{BFJ1}. 
Note that we do not assume here that $X$ and the model $\Xcal$ are normal. This leads to additional complications. In case of a normal model $\Xcal$ (as in {\it{loc. cit.}}), 
the finite morphisms $g_m$ are the identity and $\mathfrak{b_m}$ is just the base-ideal $\mathfrak{a_m}$ of $\Lcal^{\otimes m}$ on 
$\Zcal_m=\Xcal$. 
Then pointwise convergence holds on $X^{\rm div}$. 
If the normalization $\Zcal$ of $\Xcal$ would be finite over $\Xcal$, then we could use $\Zcal_m = \Zcal$ for all $m$.

{After we have submitted this paper, Boucksom and Eriksson \cite[Theorem 4.20]{BE} showed that the normalization $\Zcal$  is indeed  finite over $\Xcal$ and so we may use $\Zcal_m=\Zcal$ in the lemma. Moreover, the pointwise convergence holds on $X^{\rm div}$. We were informed by Ofer Gabber that the finiteness of the normalization $\Zcal$ over $\Xcal$ was also shown by Anantharaman in \cite[Th\'eor\`eme 1' in Appendice II]{Anantharaman}. We thank Ofer Gabber very much for providing us with this reference.}

\begin{proof}
We have 
$\mathfrak{a}_m \cdot \mathfrak{a}_l \subset \mathfrak{a}_{m+l}$ by definition of 
the base-ideal $\mathfrak{a_m}$ of $\Lcal^{\otimes m}$.
It follows that the sequence $(\log|\mathfrak{a}_m|)$ is super-additive, i.e.
$$ \log|\mathfrak{a}_{m+l}| \geq \log|\mathfrak{a}_m| + \log|\mathfrak{a}_l| $$
for all $m,l \in \N$. 
For $m$ sufficiently large, $\mathfrak{a}_m$ is a coherent {vertical} ideal sheaf and hence $\log|\mathfrak{a}_m|>-\infty$.
By Fekete's super-additivity lemma, the limit of the sequence $\frac{1}{m}\log|\mathfrak{a}_m|$ exists pointwise in $]-\infty, \infty]$. Since $\mathfrak{a}_m$ is an  ideal sheaf, we have $\frac{1}{m}\log|\mathfrak{a}_m| \leq 0$ anyway and so Fekete's super-additivity lemma gives in fact 
\begin{equation} \label{application of Fekete}
 -\infty < \lim_{m \to \infty} \frac{1}{m}\log|\mathfrak{a}_m|  = \sup_m  \frac{1}{m}\log|\mathfrak{a}_m| \leq 0
\end{equation}
pointwise on $\Xan$. 

We choose 
 $\ve >0$ in the value group $\Gamma$ and we set $\theta := c_1(L,\metr)$. 
Now we use that $\metr$ is the pointwise limit of semipositive model metrics on $L$ over $X^{\rm div}$. 
This is equivalent to the property that $0$ is the pointwise limit of $\theta$-psh model-functions over $X^{\rm div}$(see \ref{semipositive metric vs psh}).
 It follows from \ref{lemma model function Cartier} that there is a vertical $\Q$-Cartier divisor $D$ on a model $\Xcal'$ of $X$ such that $\varphi_D$ is $\theta$-psh and such that 
\begin{equation} \label{pointwise inequalities}
\varphi_D(x) \geq -\ve, \quad 
\varphi_D(y) \leq \ve 
\end{equation}
for all $x \in S$ and all divisorial points $y$ associated with $\Xcal$. 
We may assume $D$ lives on a $\kcirc$-model $\Xcal'$ with  a morphism $\pi:\Xcal' \rightarrow \Xcal$ extending the identity on $X$.  
By Proposition \ref{theta-psh and skeleton}, we get $\varphi_D \leq \varepsilon$.

Let us consider the model function $\varphi_{D'} \coloneqq \varphi_D - \ve$ on $X$ 
with associated vertical $\Q$-Cartier divisor $D'$ on $\Xcal'$.
We conclude that 
\begin{equation} \label{varphi' inequalities}
\varphi_{D'}(x) \geq -2 \ve, \quad \varphi_{D'} \leq 0
\end{equation}
for all $x \in S$. 
We note that  $\Ocal(D') \cong \Ocal(D)$ as $\Q$-line bundles (which means that $D$ and $D'$ are $\Q$-linearly equivalent) and hence 
$\pi^*(\Lcal) \otimes \Ocal(D')$ is nef using that $\varphi_D$ is $\theta$-psh. Let $\theta'$ be  the corresponding semipositive closed $(1,1)$-form on $X$. Since $\{\theta'\}=\{\theta\}$ is ample, we may apply Proposition \ref{semipositive vs ample} to deduce that there is a sufficiently large $\kcirc$-model $\Xcal''$ dominating $\Xcal'$ and a model function $\varphi''$ with
\begin{equation} \label{varphi'' inequalities}
 -\ve \leq \varphi''\leq 0
\end{equation}
such that $\theta' + dd^c\varphi''$ is $\Xcal''$-positive. 
{Let $D''$ be the  vertical $\Q$-Cartier divisor on $\Xcal''$ such that $\varphi'' = \varphi_{D''}$.}

 To ease notation, we may assume that $\Xcal' = \Xcal''$. Then we deduce that $\pi^*(\Lcal) \otimes \Ocal(D'+D'')$ is an ample $\Q$-line bundle. 
Using that $K$ is algebraically closed, the reduced fibre theorem \cite[Theorem 2.1']{BLR4}  shows that there is a proper model $\Ycal$ of $X$ dominating $\Xcal'$ with reduced special fibre. 
By Lemma \ref{model functions and multiplicities}, \eqref{varphi' inequalities} and \eqref{varphi'' inequalities}, the pull-backs $-E',-E''$ of $-D',-D''$ to $\Ycal$ are both effective vertical $\Q$-Cartier divisors. Let $\rho:\Ycal \to \Xcal$ be the morphism extending the identity on the generic fibre. 
We note that $\rho^*(\Lcal) \otimes \Ocal(E'+E'')$ is semiample which means that $\rho^*(\Lcal)^{\otimes m_0} \otimes \Ocal(m_0(E'+E''))$ is a honest line bundle on $\Ycal$ generated by global sections for a suitable $m_0 \in \N \setminus \{0\}$.

Since models are proper over $\kcirc$, any $\kcirc$-morphism between models is proper and hence we may consider the Stein factorization $\rho= g \circ \rho'$ as in \ref{Stein factorization} {for morphisms $g:\Zcal \to \Xcal$ and $\rho':\Ycal \to \Zcal$ of schemes over $\kcirc$.}
Similarly as in Step 2 of the proof of Lemma \ref{lemma limit psh function}, we deduce that 
$g: \Zcal \to \Xcal$ is a finite morphism of $\kcirc$-models of $X$ extending $\id_X$. By (3) in \ref{Stein factorization}, we have
\begin{equation} \label{sheaf push-forward trivial}
 \rho'_*(\Ocal_\Ycal) = \Ocal_\Zcal.
\end{equation}
By the projection formula and \eqref{sheaf push-forward trivial}, we get 
$\rho'_*(\rho^*(\KL^{\otimes m})) \cong g^*(\KL^{\otimes m})$ and hence
\begin{equation} \label{consequence of projection formula}
H^0(\Zcal, g^*(\Lcal^{\otimes m})) = H^0(\Ycal,\rho^*(\Lcal^{\otimes m})). 
\end{equation}
For all $m \in \N$ divisible by $m_0$, we have seen that $-mE'-mE''$ is an effective vertical Cartier divisor and hence 
$\Ocal(mE'+mE'')$ is a vertical ideal sheaf in $\Ocal_{\Ycal}$. We get a canonical inclusion
\begin{equation} \label{inclusion of sections}
 \rho^*(\Lcal)^{\otimes m} \otimes \Ocal(m(E'+E'')) \subset \rho^*(\Lcal)^{\otimes m} 
\end{equation}
for all $m \in \N$ divisible by $m_0$. The left hand side is globally generated.
Let $\mathfrak{b}_m$ be the base ideal of $g^*(\Lcal^{\otimes m})$ on $\Zcal$. 
We claim that 
\begin{equation} \label{inclusion into the base-ideal}
\Ocal(m(E'+E'')) \subset \Ocal_\Ycal \mathfrak{b}_m \subset \Ocal_\Ycal.
\end{equation}
Note that the inclusion $\Ocal(m(E'+E'')) \subset \Ocal_\Ycal$ is given by multiplication with the canonical global section $s_{-m(E'+E'')}$ of $\Ocal(-mE'-mE'')$. We check \eqref{inclusion into the base-ideal} at $y \in \Ycal$. Using semiampleness, there is a global section $s$ of $\rho^*(\Lcal)^{\otimes m} \otimes \Ocal(m(E'+E''))$ which does not vanish at $y$, i.e.\  $s^{-1}$ is a local section at $y$.  Let $t$ be any section of $\Ocal(m(E'+E''))$ around $y$. We have to show that $t \otimes s_{-m(E'+E'')}$ is a section of $\Ocal_\Ycal \mathfrak{b}_m$ around $y$. To see this, we write 
$$ t \otimes s_{-m(E'+E'')} = (s \otimes s_{-m(E'+E'')}) \otimes (t \otimes s^{-1}).$$
Since $-mE'-mE''$ is an effective vertical Cartier divisor, $ s \otimes s_{-m(E'+E'')} $ is a global section of $\rho^*(\Lcal^{\otimes m})$ and hence it is the pull-back of a global section of $g^*(\Lcal^{\otimes m})$ by \eqref{consequence of projection formula}. 
Moreover, $t \otimes s^{-1}$ is a local section of $\rho^*(\Lcal^{-m})$ around $y$ and hence it is an $\Ocal_\Ycal$-multiple of the pull-back of a local section of $g^*(\Lcal^{-m})$ at $\rho'(y)$. 
It follows from the definition of the base-ideal $\mathfrak{b}_m$ that $t \otimes s_{-m(E'+E'')}$ is a local section of 
$\Ocal_\Ycal \mathfrak{b}_m$ around $y$ proving \eqref{inclusion into the base-ideal}.

It follows from \eqref{varphi' inequalities}, \eqref{varphi'' inequalities} and \eqref{inclusion into the base-ideal} that
$$-3 \ve \leq \varphi_{E'+E''}(x) \leq \frac{1}{m} \log|\mathfrak{b}_m|(x)$$ 
for all $m \in \N$ divisible by $m_0$ and all $x \in S$. 
By \eqref{application of Fekete} 
applied to the base ideals $\mathfrak b_m$ on $\Zcal$ instead of $\mathfrak a_m$, we deduce that  
\begin{equation} \label{pointwise inequality}
-3 \varepsilon \leq \lim_{m \to \infty} \frac{1}{m} \log|\mathfrak{b}_m|(x) \leq 0
\end{equation}
for all $x \in S$. 
Using a sequence $\varepsilon \to 0$, we construct easily from  \eqref{pointwise inequality} a sequence of finite morphisms $g_m:\Zcal_m \to \Xcal$ with the required property.
\end{proof}

The following result is similar to step 2 in the proof of \cite[Lemma 5.12]{BFJ1}. 
Note that we need here another argument as  the multiplier ideals  used in \cite{BFJ1} do  not work in residue characteristic $p>0$.
Let us recall that a line bundle $L$ on a scheme is called \emph{semiample} if $L^{\otimes m}$ is globally generated for some $m\in \N_{>0}$.

\begin{lem} \label{Lemma 5.12-second step}
Let $L$ be a semiample line bundle on the projective variety $X$ over the algebraically closed non-archimedean field $K$ with a $\kcirc$-model $\Lcal$ on the algebraic $\kcirc$-model $\Xcal$ of $X$.  
Suppose that for any finite $S \subset X^{\rm div}$, there is a sequence of 
algebraic $\kcirc$-models $\Zcal_m$ of $X$ with $\id_X$ extending to 
finite morphisms $g_m:\Zcal_m \to \Xcal$ 
such that $\frac{1}{m}{\log |\mathfrak b_m|}$ converges pointwise to $0$ on $S$, where $\mathfrak b_m$ is the base-ideal of $g_m^*(\Lcal)^{\otimes m}$ as before. Then $\metr_\Lcal$ is a semipositive model metric.
\end{lem}

\begin{proof}
{In this proof, we will need intersection theory on  $\kcirc$-models. Since the base $\kcirc$ is not noetherian, we will use the intersection theory with Cartier divisors from \cite{gubler98} (see also \cite[Section 2]{GS15} and \cite[Appendix]{gubler-rabinoff-werner} for algebraic versions). The main ingredient is that every vertical Cartier divisor $D$ has an associated Weil divisor $\cyc(D)$ with multiplicities in  the value group $\Gamma$. To define the multiplicities, we pass to a dominating model with reduced special fibre and use the projection formula
(see \cite[3.8, 3.10]{gubler98}). In the algebraic setting, such a dominating model exists by the reduced fibre theorem \cite[Theorem 2.1']{BLR4}.} 

Let $n\coloneqq \dim(X)$. 
Hence $\Xcal$ is irreducible of dimension $n+1$. 
We choose a closed curve $Y$ in the special fibre $\Xcal_s$. Then we have to show that $\deg_\Lcal(Y) \geq 0$. 
We follow the strategy  of \cite{Goo69} to use the blow-up $\pi:\Xcal' \rightarrow \Xcal$ along $Y$ (as suggested in \cite[Remark 5.13]{BFJ1}). Then $E:=\pi^{-1}(Y)$ is an effective Cartier divisor on $\Xcal'$ which is vertical. 
Note that any $\kcirc$-model of $X$ is dominated by a projective $\kcirc$-model of $X$ \cite[Proposition 10.5]{gubler03} and so we may replace 
 $\Xcal'$  by a projective dominating model. 
 Then we have a very ample invertible sheaf $\Hcal'$ on $\Xcal'$. 
 {We may view the vertical closed subscheme $E$ of $\Xcal'$ as a projective scheme of pure dimension $n$ over the residue field $\ktilde$ and we consider the surjective morphism $E \to Y$ induced by $\pi$. Then the support of the Weil divisor $\cyc(E)$ is contained in $E$. Using a dominating model with reduced special fibre, using \cite[Proposition A.7]{gubler-rabinoff-werner} and the projection formula \cite[Proposition 4.5]{gubler98}, it is clear that $\cyc(E)$ has a component mapping onto $Y$.} 
 It follows from using generic hyperplane sections and the fibre theorem \cite[Exercise II.3.22]{Har} 
that $\pi_*({c_1(\Hcal')^{n-1}.\cyc(E)})$ is a positive multiple of $Y$. 
By the projection formula,  it is enough to show 
\begin{equation} \label{degree claim on blow up}
 \deg_{\Lcal'}(c_1(\Hcal')^{n-1}.\cyc(E)) \geq 0
\end{equation}
for $\Lcal':=\pi^*(\Lcal)$. 

Let $S$ be the set of divisorial points of $\Xan$ associated with the $\kcirc$-model $\Xcal'$. 
By \ref{divisorial points associated to algebraic model} the set $S$ is finite.
{By our standing assumptions in Lemma \ref{Lemma 5.12-second step}, there is} a sequence of finite morphisms $g_m:\Zcal_m \to \Xcal$ with base-ideal $\mathfrak b_m$ {of $g_m^*(\Lcal)^{\otimes m}$ such that $\frac{1}{m}{\log |\mathfrak b_m|}$ converges pointwise to $0$ on $S$.} 
The crucial new idea is to consider a sequence of morphisms $\psi_m:\Xcal_m \rightarrow \Xcal'$ related to the base-ideals $\mathfrak{b}_m$. In the following, $m$ is a sufficiently {divisible} integer such that the base-ideal $\mathfrak b_m$ is vertical. 
Let $\pi':\Zcal_m' \to \Zcal_m$ be the base change of $\pi$ to $\Zcal_m$ and 
let $\psi_m':\Xcal_m \rightarrow \Zcal_m'$ be the blow up of $\Zcal_m'$ in the closed subscheme $(\pi')^{-1}(V(\mathfrak{b}_m))$. Then we have a commutative diagram 
\[ 
\xymatrix{
\Xcal_m \ar[rd]_{\psi_m} \ar[r]^{\psi_m'} & \Zcal_m' \ar[r]^{\pi'}  \ar[d]^{g_m'} &  \ar[d]^{g_m} \Zcal_m \\
& \KX' \ar[r]^\pi & \KX
}
\] 
of morphisms of $\kcirc$-models of $X$ extending $\id_X$. Note that the base change $g_m'$ of $g_m$ is a finite morphism.
Setting $\pi_m' \coloneqq \pi' \circ \psi_m'$, we have an effective vertical Cartier divisor $D_m \coloneqq (\pi_m')^{-1}(V(\mathfrak{b}_m))$ on $\Xcal_m$ 
and we denote by $s_{-D_m}$ the canonical meromorphic section of $\mathcal{O}(-D_m)$.
We define $\pi_m := \pi \circ \psi_m$. 
Note that $E_m:=\pi_m^{-1}(Y)=\psi_m^{-1}(E)$ is an effective Cartier divisor on $\Xcal_m$ and that $\Hcal_m := \psi_m^*(\Hcal')$ is 
a line bundle on $\Xcal_m$ which is generated by global sections. 
We conclude from refined intersection theory that 
\begin{equation} \label{1-dim representative} 
 {\rm cl}(C)=c_1(\Hcal_m)^{n-1}.\cyc(E_m) \in {\rm CH}_1(E_m)
\end{equation}
for an effective $1$-dimensional cycle $C$ of $\Xcal_m$ with support over $Y$. We consider the invertible sheaf
$\Lcal_m := \pi_m^*(\Lcal^{\otimes m})\cong \psi_m^*(\Lcal'^{\otimes m}) 
\cong (\pi_m')^*(g_m^*(\Lcal^{\otimes m}))$
of $\Xcal_m$. 
We claim that 
\begin{equation} \label{crucial inequality}
\deg_{\Lcal_m}(C) \geq \deg_{\Ocal(D_m)}(C).
\end{equation}
To prove this, let $C_m$ be any irreducible component of $C$. We choose $\zeta_m \in C_m$ and let $\zeta := \pi_m'(\zeta_m)$. 
We note first that the stalk of $\Lcal_m(-D_m)$ at $\zeta_m$ is generated by global sections. Indeed, it follows from the definitions that there is a global section $s_m$ of $g_m^*(\Lcal^{\otimes m})$ and an invertible section $\ell_m$ of $g_m^*(\Lcal^{\otimes m})$ at $\zeta$ such that 
$(\pi_m')^*(s_m/\ell_m)$ is an equation of the Cartier divisor $D_m$ at $\zeta_m$.
Using that $D_m = \pi_m'^{-1}(V(b_m))$, a similar local consideration in any point of $\Xcal_m$ shows that 
  $t_m:= (\pi_m')^*(s_m) \otimes s_{-D_m}$ is a global section of $\Lcal_m(-D_m)$ and the choice of $s_m$ yields that $t_m$ generates the stalk at $\zeta_m$. We deduce that the restriction of $t_m$ to $C_m$ is a global section which is not identically zero
 and hence 
$$\deg_{\Lcal_m}(C_m) = \deg_{\Ocal(D_m)}(C_m) + \deg(\Div(t_m|_{C_m})) \geq \deg_{\Ocal(D_m)}(C_m)$$
proving \eqref{crucial inequality}. 
By the projection formula \cite[Proposition 4.5]{gubler98} and \eqref{1-dim representative}, we have 
$$ m  \deg_{\Lcal'}(c_1(\Hcal')^{n-1}.E) = \deg_{\Lcal_m}(C)$$
and hence \eqref{crucial inequality} leads to 
$$ m \deg_{\Lcal'}(c_1(\Hcal')^{n-1}.E) \geq \deg_{\Ocal(D_m)}(C)
=\deg_{\Ocal(D_m)}(c_1(\Hcal_m)^{n-1}.\cyc(E_m)).$$
Commutativity of intersection product \cite[Theorem 5.9]{gubler98} shows
\begin{equation} \label{application of commutativity}
 m \deg_{\Lcal'}(c_1(\Hcal')^{n-1}.\cyc(E)) \geq \deg(c_1(\Hcal_m)^{n-1}. E_m. \cyc(D_m)).
\end{equation}
As we may replace $\Xcal_m$ in the above considerations by any dominating $\kcirc$-model of $X$, the reduced fibre theorem \cite[Theorem 2.1']{BLR4} shows that we may assume that $\Xcal_m$ has reduced special fibre. 
We have 
$$\cyc(D_m) = \sum_{W} \mu_W W,$$
where $W$ ranges over all irreducible components of the special fibre of $\Xcal_m$.
Since the special fibre of $\Xcal_m$ is reduced, \cite[Lemma 3.21]{gubler98} shows that
there is a unique point $\xi_W$ of the analytification $\Xan$ of the generic fibre of $\Xcal_m$ with reduction equal to the generic point of $W$ 
and the multiplicities $\mu_W$ are given by 
$$\mu_W = -\log \|s_{D_m}(\xi_W)\|_{\Ocal(D_m)}.$$
We insert this in \eqref{application of commutativity} and use again projection formula to get
\begin{equation} \label{back to X'}
 m \deg_{\Lcal'}(c_1(\Hcal')^{n-1}.\cyc(E)) \geq \sum_V \sum_{W: \psi_m(W)=V} \mu_W [W:V] \deg(c_1(\Hcal')^{n-1}. E. V),
\end{equation}
where $V$ ranges over all irreducible components of $(\Xcal')_s$ and $W$ ranges over the irreducible components of $(\Xcal_m)_s$ with $\psi_m(W)=V$. Here, $[W:V]$ is the degree of the induced map $W\rightarrow V$. Note that $\xi_W=\psi_m(\xi_W)$ is a divisorial point of $\Xan$ which reduces to the generic point of $V$ in the model $\Xcal'$ 
and hence $\xi_W$ is an element of the  set $S$ of divisorial points of $\Xan$ associated with the $\kcirc$-model $\Xcal'$. 

We choose  $\ve>0$ small. 
By the convergence assumption on the base-ideals $\mathfrak b_m$ and using that $S$ is finite, there is a sufficiently {divisible}   $m$ such that 
$$0 \leq -\frac{1}{m}\log |\mathfrak{b}_m|(\xi_W) \leq \ve $$
for all $W$ as above.  We conclude that
\begin{equation} \label{multiplicity estimate}
0 \leq \mu_W = -\log \|s_{D_m}(\xi_W)\|_{\Ocal(D_m)}= - \log |\mathfrak{b}_m|(\xi_W) \leq  m\ve
\end{equation}
for all $V$ and $W$ as above with  $\psi_m(W)=V$. Let $-R$ be the minimum of the finitely many  intersection numbers $\deg(c_1(\Hcal')^{n-1}. E. V)$ and $0$. 
Then \eqref{back to X'} and \eqref{multiplicity estimate} lead to 
$$\deg_{\Lcal'}(c_1(\Hcal')^{n-1}.\cyc(E)) \geq -R \ve \sum_V \sum_{W: \psi_m(W)=V} [W:V].$$
By the projection formula for $\psi_m$ applied to the Cartier divisor $\Div(\rho)$ on $\Xcal'$  for any non-zero $\rho$ in the maximal ideal of 
$\kcirc$ and using that the special fibre of $\Xcal_m$ is reduced,
we deduce easily that
$$\sum_{W: \psi_m(W)=V} [W:V] =   m_V$$
for the multiplicity $m_V$ of $(\Xcal')_s$  along $V$. 
We conclude that 
$$\deg_{\Lcal'}(c_1(\Hcal')^{n-1}.\cyc(E)) \geq -R \ve \sum_V m_V.$$
The numbers $R$ and $m_V$ are independent of $\ve$. This proves \eqref{degree claim on blow up} and hence the claim. 
\end{proof}

In the following, we use the notation introduced in \S \ref{section psh model functions}. Recall  that $\mathcal{D}(X)$ denotes the space of model functions on $X$.

\begin{thm} \label{pointwise convergence vs uniform convergence}
Let $X$ be a proper  scheme over $K$ and let $\theta$ be a closed $(1,1)$-form on $X$. Then the set of $\theta$-psh model functions is closed in $\mathcal{D}(X)$ with respect to pointwise convergence on $X^{\rm div}$.
\end{thm}

This is a generalization of Theorem 5.11 in \cite{BFJ1} as we allow 
$K$ to be an arbitrary non-archimedean field and also because we allow any proper scheme $X$. 

\begin{proof}
We may check semipositivity for the pull-back with  respect to a proper surjective morphism $X' \to X$ 
{by Proposition \ref{pullback of semipositive formal} (b).} 
 Using Chow's lemma and Proposition \ref{generically finite morphism}, we conclude that we may assume $X$ projective.

Let $\varphi$ be a model function on $X$ which is the pointwise limit over $X^{\rm div}$ of $\theta$-psh model functions on $\Xan$. 
Replacing $\theta$ by $\theta + dd^c \varphi$, we may assume that $\varphi = 0$. 
 Then the existence of a $\theta$-psh model function {$\psi$} yields that $\theta {+dd^c \psi}$ is semipositive and hence $\{\theta\}$ is nef 
(see \ref{vertical nef and generic fibre}). 
Let $\Xcal$ be an algebraic $\kcirc$-model of $X$ such that $\theta$ is determined on $\Xcal$. Then the restriction of $\theta_\Xcal$ to $X$ is nef. 
Since any $\kcirc$-model of $X$ is dominated by a projective $\kcirc$-model of $X$ \cite[Proposition 10.5]{gubler03}, we may assume that $\Xcal$ is projective.

The proof of Proposition \ref{semipositive vs ample} shows that $N^1(\Xcal/S)$ is a finite dimensional $\R$-vector space as we can see it as a subspace of $N^1(\Xcal_s)$. 
We have also seen that the ample cone in $N^1(\Xcal/S)$ is the intersection of $N^1(\Xcal/S)$ with the ample cone in $N^1(\Xcal_s)$ and hence it is open in $N^1(\Xcal/S)$. We conclude that  
there are $\Hcal_1, \dots , \Hcal_n$ ample line bundles on $\Xcal$ such that their numerical classes $\alpha_j$ form a basis of $N^1(\Xcal/S)$. Then there are $\lambda_j \in \R$ such that 
$\sum_j \lambda_j c_1(\Hcal_j) \in \Pic(\Xcal)_\R$ 
represents $\theta$.  Let $\ve_j$ be small positive numbers such that the numbers $\lambda_j + \ve_j$ are rational. We consider the $\Q$-line bundle
$$\Lcal_\ve := \bigotimes_j \Hcal_j^{\otimes (\lambda_j + \ve_j)}$$ 
on $\Xcal$ and let $L_\ve := \Lcal_\ve|_X$. 
Since $\{\theta\}$ is nef and  $\ve_j > 0$, it follows that  $L_\ve$  is ample. For any model function $\psi$ on $X$, we have 
$$c_1(L_\ve,e^{-\psi}\metr_{\Lcal_\ve}) = dd^c \psi + \theta + \sum_j \ve_j \alpha_j.$$
We conclude that a $\theta$-psh model function $\psi$ yields a semipositive model metric $e^{-\psi}\metr_{\Lcal_\ve}$. 
Since $\varphi=0$ is the pointwise limit over $X^{\rm div}$ of $\theta$-psh model functions $\psi$ on $X$,  we deduce that $\metr_{\Lcal_\ve}$ is the pointwise limit over $X^{\rm div}$ of semipositive model metrics on $L_\ve$. It follows from Proposition \ref{Lemma 5.12} that   $\metr_{\Lcal_\ve}$  is semipositive. This means that $\Lcal_\ve$ is  nef. 

By definition of  nef and using $N^1(\Xcal/S) \subset N^1(\Xcal_s)$, we see that the cone in $N^1(\Xcal/S)$ of  nef classes is the intersection of $N^1(\Xcal/S)$ with the nef cone in $N^1(\Xcal_s)$. In particular, the cone of  nef classes is closed in $N^1(\Xcal/S)$. 
Using $\ve = (\ve_1, \dots, \ve_n) \to 0$, we deduce that $\sum_j \lambda_j c_1(\Hcal_j) $   is  nef.
Since the latter represents $\theta$, we conclude that $\varphi = 0 $ is $\theta$-psh. 
\end{proof}

\begin{rem} \label{proof of the main theorem in intro}
Note that Theorem \ref{pointwise convergence of metrics} is 
a special case of Theorem \ref{pointwise convergence vs uniform convergence} by using \ref{semipositive metric vs psh}.
\end{rem}

\appendix

\section{Divisorial points} \label{appendix on divisorial points}

\begin{definition} \label{definition of a divisorial point}
Let $V$ be  a paracompact strictly $K$-analytic space. A \emph{divisorial point} $x$ of $V$ is a point $x \in V$ such that there is a formal $\kcirc$-model $\fV$ with reduction map $\pi:V \to \fV_s$ such that $\pi(x)$ is the generic point of an irreducible component of $\fV_s$. 
We call $x$ also a \emph{divisorial point} associated with the model $\fV$.
\end{definition}

\begin{art} \label{Shilov boundary}
Let $V$ be a strictly $K$-affinoid space. We recall the following facts from \cite[Proposition 2.4.4]{berkovich-book}: 
The {\it Shilov boundary of $V$} is the unique minimal closed subset $\Gamma$ of $V$ with the property that $\max_{x \in \Gamma} |f(x)| = \max_{x \in V} |f(x)| $ for every $f \in \Acal:=\Ocal(V)$. Note that $\Acal$ is a strictly $K$-affinoid algebra and let 
 $\tilde\Acal := \{a \in \Acal \mid |a|_{\rm sup} \leq 1\}/\{a \in \Acal \mid |a|_{\rm sup} < 1\}$ be its canonical reduction.
There is a \emph{canonical reduction map} $V \to \Spec(\tilde{\Acal})$ which is surjective. The generic point of an irreducible component $E$ of $\Spec(\tilde{\Acal})$ has a unique preimage  in $V$ denoted by $\xi_E$. 
The Shilov boundary $\Gamma$ is equal to the finite set of points $\xi_E$ with $E$ ranging over all irreducible components of the canonical reduction $\Spec(\tilde{\Acal})$.
\end{art}

\begin{prop} 
\label{divisorial points associated to formal affine model}
Let $\fV= \Spf(A)$ be a formal affine $\kcirc$-model of the strictly $K$-affinoid space $V$. Then the set of divisorial points of $V$ associated with $\fV$ is equal to the Shilov boundary of $V$. In particular, this set is finite.
\end{prop}

\begin{proof}
Let $\Acal := A \otimes_\kcirc K$ be the associated strictly $K$-affinoid algebra. By \cite[Proposition 2.12]{GRW2}, the canonical morphism $\iota:\Spec(\tilde\Acal) \to \fV_s$ is finite and surjective. 
Let $\pi:V \to \fV_s$ be the reduction map of $\fV$ and let $\pi': V \to \Spec(\tilde\Acal)$ be the canonical reduction map of $V$. Using that $\pi=\pi' \circ \iota$, we conclude for $x \in V$ that $\pi(x)$ is the generic point of an irreducible component of $\fV_s$ if and only if $\pi'(x)$ is the generic point of an irreducible point of $\Spec(\tilde\Acal)$. 
It follows that the set of divisorial points associated with $\fV$ is given by the points $\xi_E$ with $E$ ranging over the irreducible components of $\Spec(\tilde{\Acal})$.
We have seen in \ref{Shilov boundary} that this set  is the Shilov boundary of $V=\fV_\eta$ proving precisely our claim.
\end{proof}

\begin{art}
\label{local dimension}
For a point $x$ of a paracompact strictly $K$-analytic space $V$, recall that $\Ocal_{V,x}$ is endowed with a canonical seminorm $p_x$ which induces a canonical absolute value on the fraction field of $\Ocal_{V,x}/\{p_x=0\}$. The completion of this fraction field is a non-archimedean field extension of $K$  denoted by $\Hcal(x)$. As in \cite[9.1]{berkovich-book}, we define $s(x)$ as the transcendence degree of the residue field of $\Hcal(x)$ over $\ktilde$.

We define $\dim_x(V)$ as the minimum of the dimensions of the strictly $K$-affinoid domains in $V$ containing $x$. 
Let us pick any strictly $K$-affinoid domain $W$ of $V$ containing $x$. Then 
\begin{equation} \label{local dimension and irreducible components}
\dim_x(V) = \max_i \dim(\Ocal(W_i))
\end{equation}
where $W_i$ ranges over the irreducible components of $W$ containing $x$ and where we use the Krull dimension of the strictly $K$-affinoid algebra $\Ocal(W_i)$ on the right.
We refer to \cite[Section 1]{Duc07} for more details and additional properties on the dimension of $K$-analytic spaces.
It follows easily from  \cite[Lemma 2.5.2]{berkovich-ihes} that
\begin{equation} \label{s and local dimension}
s(x) \leq \dim_x(V).
\end{equation}
\end{art} 

\begin{prop} \label{criterion for divisorial points}
Let $x$ be a point of a paracompact strictly $K$-analytic space $V$. Then $x$ is a divisorial point of $V$ if and only if $s(x)=\dim_x(V)$.
\end{prop}
Let us remark that the result and its proof are similar to \cite[Corollaire 4.18]{Poi13}.
\begin{proof}
Let us prove that if  $s(x)=\dim_x(V)$ then $x$ is a divisorial point 
(the other implication follows easily from the definition of divisorial points, see for instance \cite[Lemme 4.4]{Poi13}). 
Let $g_1,\ldots,g_n $ be  elements in the residue field of $\Hcal(x)$ 
which are algebraically independent over $\tilde{K}$ where $n=\dim_x(V)$.  
Let $U$ be an $n$-dimensional strictly $K$-affinoid domain in $V$ containing $x$ and let $\Acal= \Ocal(U)$ be the corresponding strictly $K$-affinoid algebra.
The residue field of ${\Hcal(x)}$ can be identified with the residue field of the fraction field of 
$\Acal / \mathfrak{p}_x$ for the prime ideal $\mathfrak{p}_x := \{a \in \Acal \mid |a(x)|=0\}$.
For each $i=1\ldots n$ we can then find some functions 
$\alpha_i, \beta_i \in \Acal$ such that $|\beta_i(x)| \neq 0$, 
$|\alpha_i(x)/\beta_i(x)|=1$ and 
such that the residue classes of $\alpha_i(x)/\beta_i(x)$ are equal to $g_i$ in the residue field of $\Hcal(x)$.
Shrinking $U$ if necessary, we can then assume that the $\beta_i$'s are invertible on $U$ (it suffices to consider 
strictly $K$-affinoid Laurent domains of the form $\{ |\beta_i| \geq r \}$).
Replacing $U$ by the Weierstrass domain $\{|\alpha_i / \beta_i|\leq 1, i=1\ldots n\}$, we can even assume that $f_i:=\alpha_i/\beta_i \in \Acal^\circ$ for $i=1,\ldots,n$. 
These functions have residue classes 
$\tilde{f_i}=g_i$ in the residue field of ${\Hcal(x)}$. 
Let us now denote by $\tilde{x}\in \text{Spec}(\tilde{\Acal})$ the canonical reduction of $x$.  
Let $\kappa(\tilde{x})$ be the residue field of $\tilde{x}$.  In the following diagram
\[\Acal^\circ \to \widetilde{\Acal} \to \kappa(\tilde{x}) \hookrightarrow \widetilde{\Hcal(x)}\]
the last map is injective.
Since the $g_i$'s are algebraically independent, it follows that the $\tilde{f_i}$'s are algebraically independent in $\kappa(\tilde{x})$.
Since $\dim( \tilde{\Acal}) = \dim(\Acal) =n$ 
(see the remark at the end of $\S$ 6.3.4 in \cite{bgr}), 
it follows that $\tilde{x}$ is a generic point of $\text{Spec}(\tilde{\Acal})$ and hence it is a divisorial point of $U$. 
According to  
\cite[Lemma 8.4.5]{bosch-lectures-2015}, there exists an admissible formal scheme  $\fV$ 
with generic fibre  $V$ such that  $U$ is the generic fibre of a formal affine open subset  $\fU$ of  $\fV$. 
It follows that $x$ is a divisorial point associated with $\fV$.
\end{proof}

\begin{cor} \label{divisorial points are G-local}
Let $x$ be a point of  a strictly $K$-affinoid domain $W$ of the paracompact strictly $K$-analytic space $V$. Then $x$ is a divisorial point of $W$ if and only if $x$ is a divisorial point of $V$.
\end{cor}

\begin{proof}
Since the invariants $s(x)$ and $\dim_x(V)$ do not change if we pass from $V$ to $W$, the claim follows from Proposition \ref{criterion for divisorial points}.
\end{proof}

\begin{prop} \label{divisorial points and base change}
Let $F$ be a non-archimedean field extension of $K$ which is a subfield of $\C_K$  and let $x$ be a point of the base change $V \hat{\otimes}_K F$ of the paracompact strictly $K$-analytic space $V$.
Let $\varphi : V \hat{\otimes}_K F \to V$ be the natural map.
Then  $\varphi(x)$ is a divisorial point of $V$ if and only if  $x$ is a divisorial point of $V \hat{\otimes}_K F$.
\end{prop}

\begin{proof} 
By \cite[Proposition 1.22]{Duc07} we have 
\begin{equation} \label{local dimension and base change}
\dim_x(V \hat{\otimes}_K F)=\dim_{\varphi(x)}(V).
\end{equation}
The equality $s(x) = s(\varphi(x))$ follows easily from   \cite[Lemma 9.1.1]{berkovich-book}.
Then the claim follows  from \eqref{local dimension and base change} and Proposition \ref{criterion for divisorial points}.
\end{proof}

\begin{rem}
\label{remark div}
In general, if $F$ is an arbitrary non-archimedean extension of $K$, with the above notations, 
it is not true that divisorial points of  $V\hat{\otimes}_K F$ are mapped to divisorial points of $V$.
For instance, let $r \in |F^*|$ with $0<r<1$ and assume that $r^n \not\in |K^*|$ for all non-zero $n\in \N$. 
If  $\mathbb{D}$ denotes the closed unit disc over $K$, then the point $\eta_r \in \mathbb{D} \hat{\otimes}_K F$ 
given by the supremum over  the closed disc of radius $r$  is a divisorial point of  $\mathbb{D} \hat{\otimes}_K F$ 
(it is a point of type 2 in $\mathbb{D} \hat{\otimes}_K F$), 
but it is mapped to a point of type  3 in $\mathbb{D}$, namely the point corresponding to the closed disc of radius $r$ in $\mathbb{D}$, which is not a divisorial point of $\mathbb{D}$.
\end{rem} 

Now we restrict to the algebraic setting. For a proper scheme $X$ over $K$, let $X^{\rm div}$ denote the set of divisorial points of $\Xan$. The next result shows that it is enough to consider algebraic models.

\begin{prop} \label{algebraic divisorial points}
Let $X$ be a proper scheme over $K$. 
Then $x \in X^{\rm div}$ if and only if there is an algebraic  $\kcirc$-model $\Xcal$ of $X$ such that $x$ is a divisorial point associated with the formal completion $\hat{\Xcal}$.
\end{prop}

\begin{proof}
It is enough to show that a divisorial point $x$ of $\Xan$  associated with a formal $\kcirc$-model $\fV$ of $V$ is also associated with $\hat{\Xcal}$  
for a suitable algebraic $\kcirc$-model $\Xcal$ of $X$. 
This follows easily from the fact 
that $\fV$ is dominated by the formal completion of an algebraic $\kcirc$-model (see Lemma \ref{lemma formal vs algebraic}).
\end{proof}

\begin{art} \label{divisorial points associated to algebraic model}
We say that $x \in \Xan$ is a {\it divisorial point associated with the algebraic $\kcirc$-model $\Xcal$} if $x$ is a divisorial point associated with the formal completion $\hat{\Xcal}$ in the sense of Definition \ref{definition of a divisorial point}. Equivalently, this means that the reduction of $x$ is a generic point of an irreducible component of the special fibre $\Xcal_s$.

Note that $\hat{\Xcal}$  has a finite covering by formal affine open subsets. It follows from Proposition \ref{divisorial points associated to formal affine model} and Corollary \ref{divisorial points are G-local} that the set of divisorial points of $\Xan$ associated with $\Xcal$ is finite.
\end{art}

\begin{prop} \label{generically finite morphism}
Let $f:X \to Y$ be a generically finite surjective morphism of proper varieties over $K$. Then we have $X^{\rm div}=f^{-1}(Y^{\rm div})$.  
\end{prop}

\begin{proof}
There is an open dense subset $U$ of $Y$ such that $f$ induces a finite surjective morphism $f^{-1}(U) \to U$. For $x \in (f^{-1}(U))^{\rm an}$ and $y := f(x)$, we note that $\Hcal(x)/\Hcal(y)$ is a finite extension. Since $\dim_x(X)=\dim(X)=\dim(Y)=\dim_y(Y)$, it follows from 
Proposition \ref{criterion for divisorial points} 
 that $x$ is a divisorial point of $\Xan$ if and only if $y$ is a divisorial point of $\Yan$. The same criterion shows that every divisorial point of $X$ (resp. $Y$) is contained in the analytification of $f^{-1}(U)$ (resp. $U$). 
\end{proof}

\begin{prop} \label{divisorial points and irreducible components}
Let $(X_i)_{i \in I}$ be the irreducible components of a proper scheme $X$ over $K$. Then we have 
$$X^{\rm div} 
=
 \bigcup_{i \in I} (X_i)^{\rm div}.$$ 
\end{prop}

\begin{proof}
It follows from \eqref{s and local dimension} and Proposition \ref{criterion for divisorial points} that the divisorial points of $\Xan$ or of any $(X_i)^{\rm an}$ are contained in the Zariski open subset of $\Xan$ consisting of those divisorial points which are contained in the analytification of only one irreducible component of $X$. Now the claim follows from the fact shown in Corollary \ref{divisorial points are G-local} that divisorial points can be checked {\rm $G$}-locally and hence Zariski-locally.
Note also that divisorial points depend only on the induced reduced structure.
\end{proof}

\bibliographystyle{alpha}

\begin{thebibliography}{BGMPS16}

\bibitem[Ana73]{Anantharaman}
Sivaramakrishna Anantharaman.
\newblock Sch\'{e}mas en groupes, espaces homog\`enes et espaces alg\'{e}briques
sur une base de dimension 1.
\newblock In {\em Sur les groupes alg\'{e}briques}, M\'{e}m. 33 of {\em Bull. Soc. Math. France}, pages 5--79.  Soc. Math. France, Paris, 1973.


\bibitem[BPR16]{baker-payne-rabinoff}
Matthew Baker, Sam Payne, and Joseph Rabinoff.
\newblock Nonarchimedean geometry, tropicalization, and metrics on curves. 
\newblock {\em Algebr. Geom.} 3(1): 63--105 (2016).

\bibitem[Ber90]{berkovich-book}
Vladimir~G. Berkovich.
\newblock {\em Spectral theory and analytic geometry over non-{A}rchimedean
  fields}, volume~33 of {\em Mathematical Surveys and Monographs}.
\newblock American Mathematical Society, Providence, RI, 1990.

\bibitem[Ber93]{berkovich-ihes}
Vladimir~G. Berkovich.
\newblock  \'{E}tale cohomology for non-{A}rchimedean analytic spaces.
\newblock {\em Inst. Hautes \'Etudes Sci. Publ. Math.}, (78):5--161 (1994),
  1993.

\bibitem[Bos14]{bosch-lectures-2015}
Siegfried Bosch.
\newblock {\em Lectures on formal and rigid geometry}, volume 2105 of {\em
  Lecture Notes in Mathematics}.
\newblock Springer, Cham, 2014.

\bibitem[BGR84]{bgr}
Siegfried~Bosch, Ulrich~G{\"u}ntzer, and Reinhold~Remmert.
\newblock {\em Non-{A}rchimedean analysis}, volume 261 of {\em Grundlehren der
  Mathematischen Wissenschaften [Fundamental Principles of Mathematical
  Sciences]}.
\newblock Springer-Verlag, Berlin, 1984.
\newblock A systematic approach to rigid analytic geometry.

\bibitem[BL93a]{bosch-luetkebohmert-1}
Siegfried Bosch and Werner L{\"u}tkebohmert.
\newblock Formal and rigid geometry. {I}. {R}igid spaces.
\newblock {\em Math. Ann.}, 295(2):291--317, 1993.


\bibitem[BL93b]{bosch-luetkebohmert2}
Siegfried Bosch and Werner L{\"u}tkebohmert.
\newblock Formal and rigid geometry. {II}. {F}lattening techniques.
\newblock {\em Math. Ann.}, 296(3):403--429, 1993.

\bibitem[BLR95]{BLR4}
Siegfried Bosch, Werner L{\"u}tkebohmert, and Michel Raynaud.
\newblock Formal and rigid geometry. {IV}. {T}he reduced fibre theorem.
\newblock {\em Invent. Math.}, 119(2):361--398, 1995.

\bibitem[BE18]{BE}
S\'ebastien~{Boucksom} and Dennis~{Eriksson}.
\newblock {Spaces of norms, determinant of cohomology and Fekete points in non-Archimedean geometry}.
\newblock \href{http://arxiv.org/abs/1805.01016v1}{\tt arXiv:1805.01016v1}, 2018.



\bibitem[BFJ15]{BFJ2}
S{\'e}bastien Boucksom, Charles Favre, and Mattias Jonsson.
\newblock Solution to a non-{A}rchimedean {M}onge-{A}mp\`ere equation.
\newblock {\em J. Amer. Math. Soc.}, 28(3):617--667, 2015.

\bibitem[BFJ16]{BFJ1}
S{\'e}bastien Boucksom, Charles Favre, and Mattias Jonsson.
\newblock Singular semipositive metrics in non-{A}rchimedean geometry.
\newblock {\em J. Algebraic Geom.}, 25(1):77--139, 2016.

\bibitem[Bou71]{BouTop}
Nicolas~Bourbaki.
\newblock {\em \'{E}l\'ements de math\'ematique. {T}opologie g\'en\'erale.
  {C}hapitres 1 \`a 4}.
\newblock Hermann, Paris, 1971.

\bibitem[BGJKM]{BGJKM}
Jos{\'e}~Ignacio Burgos~Gil, Walter~{Gubler}, Philipp~{Jell}, Klaus~{K\"unnemann}, and Florent~{Martin}.
\newblock {Differentiability of non-archimedean volumes and non-archimedean
  Monge-Amp\`ere equations (with an appendix by Robert Lazarsfeld)}.
\newblock \href{http://arxiv.org/abs/1608.01919}{\tt arXiv:1608.01919}, 2016.


\bibitem[BMPS16]{BMPS16}
Jos{\'e}~Ignacio Burgos~Gil, Atsushi Moriwaki, Patrice Philippon, and
  Mart{\'{\i}}n Sombra.
\newblock Arithmetic positivity on toric varieties.
\newblock {\em J. Algebraic Geom.}, 25(2):201--272, 2016.

\bibitem[BPRS15]{BPRS15}
Jos{\'e}~Ignacio Burgos~Gil, Patrice~{Philippon}, Juan~{Rivera-Letelier}, and Mart\'in~{Sombra}.
\newblock {The distribution of Galois orbits of points of small height in toric
  varieties}.
\newblock \href{http://arxiv.org/abs/1509.01011}{\tt arXiv:1509.01011}, 2015.

\bibitem[BPS14]{BurgosPhilipponSombra}
Jos{\'e}~Ignacio Burgos~Gil, Patrice Philippon, and Mart{\'{\i}}n Sombra.
\newblock Arithmetic geometry of toric varieties. {M}etrics, measures and
  heights.
\newblock {\em Ast\'erisque}, (360):vi+222, 2014.

\bibitem[BPS15]{BPS15}
Jos{\'e}~Ignacio Burgos~Gil, Patrice Philippon, and Mart{\'{\i}}n Sombra.
\newblock Successive minima of toric height functions.
\newblock {\em Ann. Inst. Fourier (Grenoble)}, 65(5):2145--2197, 2015.

\bibitem[BPS16]{BPS16}
Jos{\'e}~Ignacio Burgos~Gil, Patrice Philippon, and Mart{\'{\i}}n Sombra.
\newblock Height of varieties over finitely generated fields.
\newblock {\em Kyoto J. Math.}, 56(1):13--32, 2016.


\bibitem[Cha06]{chambert-loir-2006}
Antoine Chambert-Loir.
\newblock Mesures et \'equidistribution sur les espaces de {B}erkovich.
\newblock {\em J. Reine Angew. Math.}, 595:215--235, 2006.

\bibitem[CD12]{chambert-loir-ducros}
Antoine~{Chambert-Loir} and Antoine~{Ducros}.
\newblock {Formes diff\'erentielles r\'eelles et courants
  sur les espaces de Berkovich}.
\newblock \href{http://arxiv.org/abs/1204.6277}{\tt arXiv:1204.6277}, 2012.

\bibitem[CT09]{CLT09}
Antoine Chambert-Loir and Amaury Thuillier.
\newblock Mesures de {M}ahler et \'equidistribution logarithmique.
\newblock {\em Ann. Inst. Fourier (Grenoble)}, 59(3):977--1014, 2009.

\bibitem[Con07]{Con07}   
Brian Conrad.
\newblock Deligne's notes on {N}agata compactifications.
\newblock {\em  J. Ramanujan Math. Soc.}, 22:205--257, 2007.


\bibitem[Duc07]{Duc07}
Antoine Ducros.
\newblock {Variation de la dimension d'un morphisme analytique p-adique}.
\newblock {\em Compositio Math.} 143(6):1511--1532, 2007.

\bibitem[EFM]{DFEM}
Lawrence~Ein, Tommaso~de~Fernex and Mircea~Mustata.
\newblock Vanishing theorems and singularities in birational geometry.
\newblock \url{http://homepages.math.uic.edu/~ein/DFEM.pdf}.

\bibitem[Fab09]{faber}
Xander Faber.
\newblock Equidistribution of dynamically small subvarieties over the function
  field of a curve.
\newblock {\em Acta Arith.}, 137(4):345--389, 2009.



\bibitem[FK18]{fujiwara-kato-1}
Kazuhiro Fujiwara and Fumiharu Kato.
\newblock {\em Foundations of rigid geometry {I}}.
\newblock EMS Monographs in Mathematics.
\newblock European Mathematical Society (EMS), Z\"{u}rich, 2018.


\bibitem[Goo69]{Goo69}
Jacob~Eli Goodman.
\newblock Affine open subsets of algebraic varieties and ample divisors.
\newblock {\em Ann. of Math. (2)}, 89:160--183, 1969.

\bibitem[Gro61]{EGAII}
Alexander~Grothendieck.
\newblock \'{E}l\'ements de g\'eom\'etrie alg\'ebrique. {II}. \'{E}tude globale
  \'el\'ementaire de quelques classes de morphismes.
\newblock {\em Inst. Hautes \'Etudes Sci. Publ. Math.}, (8):222, 1961.

\bibitem[Gro66]{EGAIV3}
Alexander~Grothendieck.
\newblock \'{E}l\'ements de g\'eom\'etrie alg\'ebrique. {IV}. \'{E}tude locale
  des sch\'emas et des morphismes de sch\'emas. {III}.
\newblock {\em Inst. Hautes \'Etudes Sci. Publ. Math.}, (28):255, 1966.


\bibitem[Gub98]{gubler98}
Walter Gubler.
\newblock Local heights of subvarieties over non-{A}rchimedean fields.
\newblock {\em J. Reine Angew. Math.}, 498:61--113, 1998.

\bibitem[Gub03]{gubler03}
Walter Gubler.
\newblock Local and canonical heights of subvarieties.
\newblock {\em Ann. Sc. Norm. Super. Pisa Cl. Sci. (5)}, 2(4):711--760, 2003.

\bibitem[Gub07a]{gubler-2007b}
Walter Gubler.
\newblock The {B}ogomolov conjecture for totally degenerate abelian varieties.
\newblock {\em Invent. Math.}, 169(2):377--400, 2007.

\bibitem[Gub07b]{gubler07}
Walter Gubler.
\newblock Tropical varieties for non-{A}rchimedean analytic spaces.
\newblock {\em Invent. Math.}, 169(2):321--376, 2007.

\bibitem[Gub08]{gubler08}
Walter Gubler.
\newblock Equidistribution over function fields.
\newblock {\em Manuscripta Math.}, 127(4):485--510, 2008.

\bibitem[Gub13]{gubler12}
Walter Gubler.
\newblock A guide to tropicalizations.
\newblock In {\em Algebraic and combinatorial aspects of tropical geometry},
  volume 589 of {\em Contemp. Math.}, pages 125--189. Amer. Math. Soc.,
  Providence, RI, 2013.

\bibitem[GK17]{GK1}
Walter Gubler and Klaus K\"unnemann.
\newblock A tropical approach to nonarchimedean {A}rakelov geometry.
\newblock {\em Algebra Number Theory}, 11(1):77--180, 2017.


\bibitem[GK15]{GK2}
Walter Gubler and Klaus K\"unnemann.
\newblock Positivity properties of metrics and delta-forms.
\newblock  \emph{J. Reine Angew. Math.}
\newblock \href{https://doi.org/10.1515/crelle-2016-0060}{\tt doi.org/10.1515/crelle-2016-0060}.

%

\bibitem[GRW16]{gubler-rabinoff-werner}
Walter~{Gubler}, Joseph~{Rabinoff}, and Annette~{Werner}.
\newblock {Skeletons and tropicalizations}.
\newblock {\em Adv. Math.}, 294:150--215, 2016.

\bibitem[GRW17]{GRW2}
Walter~{Gubler}, Joseph~{Rabinoff}, and Annette~{Werner}.
\newblock {Tropical Skeletons}.
\newblock {\em Ann. Inst. Fourier (Grenoble)}, 67(5):1905--1961, 2017.

\bibitem[GS15]{GS15}
Walter Gubler and Alejandro Soto.
\newblock Classification of normal toric varieties over a valuation ring of
  rank one.
\newblock {\em Doc. Math.}, 20:171--198, 2015.

\bibitem[Har77]{Har}
Robin Hartshorne.
\newblock {\em Algebraic geometry}.
\newblock Graduate Texts in Mathematics, No. 52.
\newblock Springer-Verlag, New York-Heidelberg, 1977.

\bibitem[Jel16]{Jell-thesis}
Philipp Jell.
\newblock {\em Differential forms on {B}erkovich analytic spaces and their
	cohomology}.
\newblock PhD thesis, Universit\"at Regensburg, 2016.
\newblock
\href{http://nbn-resolving.de/urn/resolver.pl?urn=urn:nbn:de:bvb:355-epub-347884}{\tt
	urn:nbn:de:bvb:355-epub-347884}.


\bibitem[dJ96]{deJong1}
Aise~Johan de~Jong.
\newblock Smoothness, semi-stability and alterations.
\newblock {\em Inst. Hautes \'Etudes Sci. Publ. Math.}, (83):51--93, 1996.

\bibitem[KRZ16]{KRZB}
Eric~{Katz}, Joseph~{Rabinoff}, and David~{Zureick-Brown}.
\newblock {Uniform bounds for the number of rational points on curves of small
  Mordell--Weil rank}.
\newblock {\em Duke Math. J.}, 165(16):3189--3240, 2016.

\bibitem[Kie72]{Kie72}
Reinhardt Kiehl.
\newblock{Ein "Descente"-Lemma und Grothendiecks Projektionssatz f\"ur nichtnoethersche Schemata}.
\newblock{ \em Math. Ann.}, 198:287-316, 1972.

\bibitem[Kle66]{Kle}
Steven L. Kleiman.
\newblock{Toward a Numerical Theory of Ampleness}.
\newblock { \em Annals of Mathematics}, Second Series, Vol. 84, No. 3 (Nov., 1966), pp. 293-344.

\bibitem[Poi13]{Poi13}
J\'er\^ome Poineau.
\newblock{Les espaces de Berkovich sont ang\'eliques}.
\newblock{ \em Bulletin de la Soci\'et\'e Math\'ematique de France} 141(2):267-297, 2013. 


\bibitem[RG71]{RG71}
Michel Raynaud and  Laurent Gruson.
\newblock  {Crit\`eres de platitude et de projectivit\'e. {T}echniques de "platification" d'un module}.
\newblock {\em Invent. Math.}, 13:1--89, 1971.   

\bibitem[{Sta}16]{stacks-project}
The {Stacks Project Authors}.
\newblock {\itshape Stacks Project}.
\newblock \url{http://stacks.math.columbia.edu}, 2016.

\bibitem[Tem00]{Tem}
Michael Temkin.
\newblock {On local properties of non-archimedean analytic spaces}.
\newblock {\em Math. Ann.}, 318:585-607, 2000.

\bibitem[Thu05]{thuillier05}
Amaury Thuillier.
\newblock {\em Th\'eorie du potentiel sur les courbes en g\'eom\'etrie
  analytique non-archim\'edienne. {A}pplications \`a la th\'eorie
  d'{A}rakelov}.
\newblock PhD thesis, Universit\'e de Rennes I, 2005.

\bibitem[Ull95]{Ulr95}
Peter Ullrich.
\newblock{The direct image theorem in formal and rigid geometry}.
\newblock{ \em Math. Ann.}, 301:69-104, 1995.

\bibitem[{Voj}07]{vojta-2007}
Paul~{Vojta}.
\newblock {Nagata's embedding theorem}.
\newblock \href{http://arxiv.org/abs/0706.1907}{\tt arXiv:0706.1907} 2007.

\bibitem[Yam13]{yamaki13}
Kazuhiko Yamaki.
\newblock Geometric {B}ogomolov conjecture for abelian varieties and some
  results for those with some degeneration (with an appendix by {W}alter
  {G}ubler: the minimal dimension of a canonical measure).
\newblock {\em Manuscripta Math.}, 142(3-4):273--306, 2013.

\bibitem[Yam16]{yamaki16}
Kazuhiko Yamaki.
\newblock Strict supports of canonical measures and applications to the
  geometric {B}ogomolov conjecture.
\newblock {\em Compos. Math.}, 152(5):997--1040, 2016.

\bibitem[Yua08]{yuan-2008}
Xinyi Yuan.
\newblock Big line bundles over arithmetic varieties.
\newblock {\em Invent. Math.}, 173(3):603--649, 2008.

\bibitem[YZ17]{yuan-zhang}
Xinyi Yuan and Shou-Wu Zhang.
\newblock The arithmetic {H}odge index theorem for adelic line bundles.
\newblock {\em Math. Ann.}, 367(3-4):1123--1171, 2017.,

\bibitem[Zha93]{zhang93}
Shouwu Zhang.
\newblock Admissible pairing on a curve.
\newblock {\em Invent. Math.}, 112(1):171--193, 1993.

\bibitem[Zha95]{zhang95}
Shouwu Zhang.
\newblock Small points and adelic metrics.
\newblock {\em J. Algebraic Geom.}, 4(2):281--300, 1995.

\end{thebibliography}

\newcommand{\etalchar}[1]{$^{#1}$}
\def\cprime{$'$}

{\small Walter Gubler, Fakult\"at f\"ur Mathematik,  Universit\"at Regensburg,
Universit\"atsstrasse 31, D-93040 Regensburg, walter.gubler@mathematik.uni-regensburg.de}

{\small Florent Martin, Fakult\"at f\"ur Mathematik,  Universit\"at Regensburg,
Universit\"atsstrasse 31, D-93040 Regensburg, florent.martin@mathematik.uni-regensburg.de}

\end{document}